\documentclass[12pt]{article}
\usepackage[hmargin={1.75truecm,2truecm},vmargin={2truecm,2truecm}]{geometry}
\usepackage[english]{babel}
\usepackage[utf8]{inputenc}
\RequirePackage{bm}
\RequirePackage{endnotes}
\usepackage{afterpage}
\usepackage{geometry}
\usepackage{multicol,float}
\usepackage{multirow}
\usepackage{hyperref}
\usepackage{amsmath}
\usepackage{amsthm}
\usepackage{amsfonts}
\usepackage{amssymb}
\usepackage{graphicx}
\usepackage{xcolor}
\usepackage{arydshln}
\usepackage[affil-it]{authblk}
\usepackage{natbib}

\usepackage{subcaption}

\usepackage{longtable}
\usepackage{svg}
\usepackage{url}
\usepackage{changepage}
\usepackage{tabularx,booktabs}
\usepackage{threeparttable}

\usepackage{xspace}
\usepackage{etex}
\usepackage{enumitem}
\usepackage{dsfont}
\usepackage{mathtools}

\usepackage{pdflscape}

\usepackage{hyperref}

\usepackage{algorithm}
\usepackage{algpseudocode}
\usepackage{tikz}

\usepackage{svg}
\usepackage{url}
\usepackage{changepage}
\usepackage{tabularx,booktabs}
\usepackage{threeparttable}

\usepackage{multirow}
\usepackage{multicol} 
\usepackage[capitalize]{cleveref}

\usepackage{pgfplots}
\usepackage{caption,subcaption}
\usepackage{textcomp}
\usepgfplotslibrary{groupplots}
\usepgfplotslibrary{units}

\usepackage{graphicx}
\usepackage{microtype} 

\usepackage[textsize=small,textwidth=3.3cm]{todonotes}

\usepackage{tikz}
\usetikzlibrary{datavisualization.formats.functions,shadows}
\usetikzlibrary{decorations.pathreplacing}
\usetikzlibrary{shapes.geometric}
\usetikzlibrary{arrows.meta}
\usetikzlibrary{petri}

\definecolor{dkgreen}{rgb}{0,0.6,0}
\definecolor{gray}{rgb}{0.5,0.5,0.5}
\definecolor{mauve}{rgb}{0.58,0,0.82}

\usepackage{pdflscape}

\usepackage{longtable}

\DeclareMathOperator*{\argmax}{arg\,max}
\DeclareMathOperator*{\argmin}{arg\,min}

\newtheorem{theorem}{Theorem}[section]

\newtheorem{definition}[theorem]{Definition}

\newtheorem{proposition}[theorem]{Proposition}
\newtheorem{lemma}[theorem]{Lemma}
\newtheorem{corollary}[theorem]{Corollary}

\usepackage{cleveref}
\Crefformat{appendix}{\S#2#1#3}
\crefname{theorem}{Theorem}{Theorems}
\crefname{example}{Example}{Examples}
\crefname{observation}{Observation}{Observations}
\crefname{remark}{Remark}{Remarks}
\crefname{proposition}{Proposition}{Propositions}
\crefname{lemma}{Lemma}{Lemmas}
\crefname{corollary}{Corollary}{Corollaries}
\crefname{algocf}{Algorithm}{Algorithms}	
\crefname{table}{Table}{Tables}	
\crefname{figure}{Figure}{Figures}
\crefname{algorithm}{Algorithm}{Algorithms}
\crefname{section}{Section}{Sections}
\crefname{algorithm}{Algorithm}{Algorithms}

\title{An efficient branch-and-cut approach for the sequential competitive facility location problem under partially binary rule}
\author[a]{Yu-Qi Guo}
\author[a]{Yan-Ru Wang}
\author[a]{Wei-Kun Chen}
\author[b]{Yu-Hong Dai}

\affil[a]{\small School of Mathematics and Statistics/Beijing Key Laboratory on MCAACI, Beijing Institute of Technology, Beijing 100081, China}
\affil[b]{\small Academy of Mathematics and Systems Science, Chinese Academy of Sciences, Beijing 100190, China}
\date{}

\DeclareMathOperator{\conv}{conv}

\newcommand{\be}{\boldsymbol{e}}

\newcommand{\LP}{\text{LP}\xspace}

\newcommand{\MINLP}{{MINLP}\xspace}
\newcommand{\MILP}{{MILP}\xspace}

\newcommand{\CFLP}{\text{CFLP}\xspace}

\newcommand{\SCFLP}{\text{SCFLP}\xspace}

\newcommand{\BnC}{\text{B\&C}\xspace}

\newcommand{\CQ}{\mathcal{Q}\xspace}
\newcommand{\CU}{\mathcal{U}\xspace}

\newcommand{\CP}{\mathcal{P}\xspace}
\newcommand{\W}{\mathcal{W}\xspace}
\newcommand{\U}{\mathcal{U}\xspace}

\newcommand{\CS}{{\mathcal{S}}}
\newcommand{\CT}{{\mathcal{T}}}

\newcommand{\CO}{{\mathcal{O}}}
\newcommand{\CZ}{\mathcal{Z}}
\newcommand{\X}{\mathcal{X}}
\newcommand{\Y}{\mathcal{Y}}
\newcommand{\F}{\mathcal{F}}

\newcommand{\R}{\mathbb{R}}

\newcommand{\proj}{{\rm{proj}}}

\newcommand{\St}{\rm s.t.}

\newcommand{\tblT}{\texttt{T}\xspace}

\newcommand{\tblN}{\texttt{N}\xspace}
\newcommand{\tblO}{\texttt{Obj}\xspace}

\newcommand{\tblUB}{\texttt{\texttt{UB}}\xspace}

\newcommand{\tblLB}{\texttt{LB}\xspace}

\newcommand{\tblTorG}{\texttt{T\,(G\%)}\xspace}
\newcommand{\tblSolved}{\texttt{Sol.}}
\newcommand{\tblAve}{\texttt{Ave.}}

\newcommand{\tblG}{\texttt{G\%}\xspace}

\newcommand{\tblrG}{\texttt{LPG\%}\xspace}

\newcommand{\tblCutTime}{\texttt{CT}\xspace}
\newcommand{\tblCutNum}{\texttt{C}\xspace}

\newcommand{\tblOldSubmodular}{\texttt{B\&C+SF}\xspace}
\newcommand{\tblNewSubmodular}{\texttt{B\&C+GSF}\xspace}
\newcommand{\tblZLP}{\texttt{B\&C+EF}\xspace}
\newcommand{\tblBiesinger}{\texttt{EA}\xspace}

\begin{document}
	\maketitle

\begin{abstract}
We investigate the sequential competitive  facility location problem (\SCFLP) under  partially binary rule where  two companies sequentially open a limited number of  facilities to maximize their market shares, requiring customers to patronize, for each company, the facility with the highest utility. 
The \SCFLP is a bilevel mixed integer nonlinear programming (\MINLP) problem and can be rewritten as a single-level \MINLP problem, where each nonlinear constraint corresponds to a  hypograph of a multiple ratio function characterizing the leader's market share for {a fixed follower's location choice}.
Through establishing the submodularity of the multiple ratio functions, we characterize the mixed 0-1 set induced by each hypograph using submodular inequalities and extend a state-of-the-art branch-and-cut (\BnC) algorithm onto the considered \SCFLP.
To address the challenge arising in the poor linear programming (\LP) relaxation of the underlying formulation, we develop  
two new mixed integer linear programming (\MILP) formulations for the \SCFLP as well as efficient \BnC algorithms based on them.
The first \MILP formulation is based on a class of improved submodular inequalities, which include the classic submodular inequalities as special cases, and {together with the trivial inequalities} are able to characterize the convex hull of the mixed 0-1 set.
The second one is an extended formulation of the first one that provides the same \LP relaxation bound.
We also develop efficient algorithms for the {separations} of the exponential families of the inequalities arising in the \MILP formulations.
By extensive computational experiments, we demonstrate that 
the proposed \BnC algorithms significantly outperform an adapted state-of-the-art \BnC algorithm and a {sophisticated} heuristic algorithm in the literature.
{Moreover}, the proposed \BnC algorithms are capable of finding optimal solutions for  \SCFLP instances with up to $1000$ customers and facilities within a time limit of two hours. 
\end{abstract}

\section{Introduction}\label{sec:introduction}
In the competitive facility location problem (\CFLP), two or multiple companies 
attempt to {open facilities} to compete 
for the market share of the customers.
The customers (or customer zones) are partially or totally served by the facilities located by the companies, according to some prespecified \emph{customer choice rule}, and the objectives of the companies are to maximize the market shares.
\CFLP{s} arise in a wide variety of  applications such as  building charging stations for 
 electric vehicles \citep{Zhao2020}, locating park-and-ride facilities \citep{Aros-Vera2013,Freire2016},
  designing preventive healthcare networks \citep{Zhang2012},  
   opening  new retail stores \citep{Mendez-Vogel2023a},
    and deploying lockers for last-mile delivery \citep{Lin2022a}.  
    For detailed reviews of \CFLP{s}, we refer to the surveys \cite{Eiselt1997},
    \cite{Plastria2001}, \cite{Suarez-Vega2004}, \cite{Kress2012}, \cite{Mishra2023},  and  \cite{Drezner2024}. 
In the literature, various \CFLP{s} have been investigated. 
Based on the competition types, \CFLP{s} can be categorized into three classes: static,  sequential, 
and dynamic \CFLP{s}. 
{For the static \CFLP,   a newcomer company attempts to enter a shared  market by opening new facilities to attract the customer demand, 
where competitors' facilities already exist and do not react to the newcomer's action.}
This problem has been widely applied in various contexts; see 
\citet{Benati2002,Aboolian2007,Aros-Vera2013,Haase2014,Freire2016,Ljubic2018,Mai2020}, and \citet{Legault2025} among many of them. 
{The static \CFLP focuses on {decisions of locating facilities made at a short-term operational level}, assuming that the facilities of competitors are {already} established and will remain unchanged for the foreseeable future.}
Different from the static \CFLP, sequential and dynamic \CFLP{s} allow
the competitors to react to the entrant who opens new facilities in the shared market.
Specifically, 
the dynamic \CFLP suggests that {the} two non-cooperative companies can simultaneously
decide where to locate their facilities, without the knowledge of each other's choice \citep{Godinho2010}. 
This process is iterated (where the two companies may relocate their facilities) and terminates when reaching the so-called \emph{Nash equilibrium} (if exists). 
A Nash equilibrium is a configuration of all variables in the model 
such that no competitor has an incentive to deviate unilaterally from the current solution \citep{Drezner2024}. 
In contrast to the dynamic \CFLP, in the sequential \CFLP, the companies will not {relocate} their facilities, as establishing facilities is usually expensive.
It can be considered as a leader-follower Stackelberg game \citep{Stackelberg1934}, 
{where the leader attempts to open facilities to maximize his/her market share, 
anticipating the follower's location choice made after the leader {takes action}.}
Many  sequential \CFLP models are available in the literature; see, for instance, \citet{Eiselt1997}, \citet{Drezner1998}, \citet{Plastria2008},  \citet{Kucukaydin2011, Kucukaydin2012}, 
 \citet{Gentile2018}, and \citet{Qi2024}. 
In this paper, we focus on the sequential \CFLP, which is long-term oriented and {relevant in the context of} locating costly chained facilities such as hotels, supermarkets, and shopping malls \citep{Qi2024}.

Another categorization of the {\CFLP{s}} is based on the patronizing behavior of customers.
The patronizing behavior of customers is based on the attractiveness/utility of open facilities,  
which typically depends on a set of attributes such as distance, transportation cost, waiting time,
 and facility  capacity. 
In the literature, several customer choice rules have been proposed.
Specifically, the binary rule assumes that each customer buys all the goods from the facility with 
 the highest utility; see  \citet{Plastria2008,Beresnev2013,Alekseeva2015,Fernandez2017}, and \cite{Lancinskas2020}.  
In contrast to the binary rule where the customer behavior is \emph{deterministic}, in the proportional rule,  customers may  patronize multiple facilities with certain probabilities. 
More specifically, this rule assumes that 
customers  split their demand across 
all open facilities of the companies in proportion  to their utilities.
The proportional rule includes the well-known multinomial  logit \citep{McFadden1974} and Huff-based gravity
 \citep{Huff1964} rules as special cases, and has been widely applied in various contexts; 
  see \citet{Benati2002,Aboolian2007,Berman2009,Kucukaydin2011,Aros-Vera2013,Haase2014,Ljubic2018,Mai2020}, and \cite{Qi2024} among  many others.
When the facilities belonging to a chain have no
attribute that makes them significantly different from each other except for, e.g., their locations, a customer tends to choose  the  store with the highest utility from each company/chain \citep{Mendez-Vogel2023a} rather than all of them.
 In this setting, we can adopt an intermediate rule known as the partially binary rule,
  originally proposed by \citet{Hakimi1990}.
This rule resembles the proportional rule but requires 
each customer to first construct a consideration set where the facility with the highest utility of each company is included, and then split all of his/her demand over the facilities in the consideration set 
in proportion to their attractiveness \citep{Suarez-Vega2004,Biesinger2016,Fernandez2017,Mendez-Vogel2023a}.
For other customer choice rules, we  refer to \citet{Lin2022a,Mendez-Vogel2023a}, and \cite{Lin2024} and the references therein. 

In this paper, we are interested in the sequential \CFLP (\SCFLP) under partially binary rule, in which the leader attempts to open $p$ $(p \geq 1)$ facilities to maximize his/her market share 
anticipating that the follower will react to the decision by opening $r$ $(r \geq 1)$ facilities to maximize his/her market share.
The problem can be formulated as a bilevel mixed integer nonlinear programming (\MINLP) problem and recasted into a single-level \MINLP problem using the technique of \citet{Qi2024}.
Our motivation is to investigate the polyhedral structure of the problem and  develop efficient exact \BnC algorithms that are capable of solving large-scale \SCFLP{s}.

\subsection{Solution approaches}

In this subsection, we review the solution approaches for solving the {\SCFLP}; {see the surveys \citet{Eiselt1997,Kress2012,Mishra2023}, and \cite{Mendez-Vogel2025} and the references therein for solution approaches of static and dynamic \CFLP{s}.}

For {\SCFLP{s}} under binary rule where each customer is assumed to be totally served by the leader or the follower, one well-known problem is the \emph{$(r|p)$-centroid problem}, {first proposed by \citet{Hakimi1983}}.
{In this problem, the leader finds the optimal locations for $p$ facilities to maximize the total demand of his/her covered customers 
 knowing that the follower will respond to the leader's action by locating $r$ new facilities to maximize the total demand of his/her covered customers.}
\citet{Plastria2008} considered a special case with $r = 1$, i.e., 
 the follower can react to the leader's location choice by placing a single new facility,
  and reformulated the bilevel optimization problem as a compact \MILP formulation, thereby enabling the use of state-of-the-art \MILP solvers to find an optimal solution for the problem.
For the general case with $r \geq 1$, the bilevel  optimization problem still admits a (single-level) \MILP reformulation but with an exponential number of constraints \citep{Roboredo2013,Alekseeva2015}. 
\citet{Roboredo2013} developed a branch-and-cut (\BnC)  algorithm to solve the $(r|p)$-centroid problem,  where the exponential number of constraints are separated on the fly, and reported computational results on instances with up to $100$ customers and $100$ facilities. 
Based on the same \MILP formulation, \citet{Alekseeva2015} developed {an iterative method} that can also find an optimal solution for the $(r|p)$-centroid problem.
For heuristic algorithms of the $(r|p)$-centroid problem, we refer to, e.g., the {tabu search algorithm \citep{Benati1994, Serra1994}}
and evolutionary algorithm \citep{Biesinger2016}.
\citet{Gentile2018} investigated three variants of the $(r|p)$-centroid problem (under binary rule) and developed \BnC algorithms to find optimal solutions for the problems.

The \SCFLP under proportional rule {was} first investigated by \citet{Drezner1998} where
the authors considered the planar problem with $p=r=1$ (i.e., both the leader and follower determine exactly one facility  in a continuous Euclidean plane)
and
{developed three heuristic algorithms, namely, the brute force, pseudo mathematical programming, and gradient search algorithms.}
\citet{ElenaSaiz2009a} considered a variant of \citet{Drezner1998} and developed a branch-and-bound approach that is able to return an $\epsilon$-optimal solution  to  the \SCFLP.
\citet{Biesinger2016} investigated the general case with arbitrary positive integers $p$ and $r$ in a discrete location space, and developed an  evolutionary algorithm to find a feasible solution for the \SCFLP.
{\citet{Qi2017} proposed a two-stage hybrid tabu search algorithm to find a feasible solution.} 
To the best of our knowledge, the first exact approach for solving \SCFLP{s} under proportional rule was provided by \citet{Qi2024}.
The authors
derived an equivalent single-level \MINLP formulation for the \SCFLP and developed a \BnC algorithm
 to the \MILP reformulation based on  submodular inequalities \citep{Nemhauser1981} and outer approximation inequalities
{\citep{Fletcher1994}} to find global optimal solutions.
In \cref{thm:decompose-Uy} of \cref{section:new-submodular}, we will develop a class of improved submodular inequalities 
for a submodular set that include the classic submodular inequalities as special case and thus can be used
 to strengthen the linear programming (\LP) relaxation of the \MILP reformulation of the \SCFLP.
{We refer to \citet{Kucukaydin2011,Kucukaydin2012,Lin2022b}, and \cite{He2025}  for investigations on variants of the
 \SCFLP that simultaneously consider the locations and attractiveness levels of the open facilities.}

The considered \SCFLP under partially binary rule  is, however, rarely studied in the literature,
 possibly due to technical challenges introduced by the highly nonlinear objective functions of the leader and follower; 
indeed, the objective functions are multiple ratio functions where both the numerator and denominator
 of each ratio term involve a nondifferentiable max function.
To the best of our knowledge, the only study was provided by \citet{Biesinger2016}, where the authors presented a bilevel \MINLP formulation, 
developed an \MILP reformulation for the follower's problem, and then used it in combination with an evolutionary algorithm
to find a feasible solution for the \SCFLP (without an optimality guarantee).
 As the  evolutionary algorithm requires to invoke off-the-shelf \MILP solvers to solve the time-consuming follower problems, only 
{computational results on instances with 100 candidate facility locations and customers were reported by \citet{Biesinger2016}.}

\subsection{Contributions}
The goal of this paper is to fill this research gap by proposing {the \emph{first} exact approach} that is capable of solving large-scale \SCFLP{s} under partially binary rule.
To achieve this, we first reformulate the bilevel \SCFLP as a single-level \MINLP problem, where each nonlinear constraint corresponds to a  hypograph of a multiple ratio function characterizing the leader's market share for a fixed follower's location choice, and then develop efficient \BnC algorithms based on linear descriptions of the mixed 0-1 set induced by each hypograph.
{In particular,}
\begin{enumerate}
	 	\item 
	 	{W}e establish the submodularity of the multiple ratio functions, which enables to   characterize the considered mixed 0-1 sets using the submodular inequalities of \citet{Nemhauser1981}, thereby obtaining an \MILP reformulation for the \SCFLP.
	 	This allows for the extension of the \BnC algorithm in \cite{Qi2024} onto the considered \SCFLP, in which the submodular inequalities are separated on the fly.
	 	
	 	\item It appears that, however,	the classic submodular inequalities are very {weak} in terms of providing a {poor} \LP relaxation bound, leading to a slow convergence of the \BnC algorithm.
	 	To address this challenge, we develop 
	 	a class of improved submodular inequalities, which generalize the classic ones as special cases, thereby obtaining another \MILP formulation with a much stronger \LP relaxation.  
		In particular, compared with the classic submodular inequalities that can only provide  a linear characterization of the considered mixed 0-1 set, the proposed improved submodular inequalities, together with the trivial inequalities, are able to provide a linear characterization of the convex hull of the considered mixed 0-1 set.
	 		  
	  \item 
	  By deriving a compact extended formulation for the convex hull of the considered mixed 0-1 set,
	  we develop an extended \MILP formulation for the \SCFLP. 
	  {Two features of the extended formulation, which {make} it able to serve as a good alternative for finding a global optimal solution to the \SCFLP, are as follows.}  
	  First, the \LP relaxation is as tight as that of the one based on the improved submodular inequalities.
	  Second, only a single linear inequality is needed to characterize the leader's market share (for a fixed follower's location choice). 
	  	This is in contrast to the one based on improved submodular inequalities where an exponential number of inequalities are needed.

\item We develop exact and heuristic algorithms for the separations of the exponential families of the above inequalities  arising in the \BnC contexts.
In particular, we show that for a fixed follower's location choice, the separation of the improved submodular inequalities can be conducted using an efficient strong polynomial-time algorithm;
and in general, {the separations of the improved submodular inequalities and the (exponential family of) inequalities in the extended formulation} can be conducted by solving the well-known $r$-median problem \citep{Hakimi1964}.

\end{enumerate}
Our extensive computational results show  that thanks to the very tight LP relaxations, the proposed \BnC algorithms
based on the improved submodular and extended formulations significantly outperform the state-of-the-art \BnC algorithm based on the submodular formulation of \cite{Qi2024} (adapted to our problem) and the evolutionary algorithm of  \citet{Biesinger2016}.
In particular, the proposed \BnC algorithms can find optimal solutions for  \SCFLP instances with up to 1000 customers and facilities within a time limit of two hours. 

\subsection{Organization and Notations}

The rest of the paper is organized as follows.  \cref{sect:problem-formulation} reviews the bilevel programming  formulation for the \SCFLP under partially binary rule and reformulates it as a single-level \MINLP problem. 
\cref{sect:milp_sub,section:new-submodular,section:ZLP} develop three \MILP formulations for the \SCFLP. 
\cref{section: implementation} discusses the algorithmic framework for the three  proposed \MILP formulations.
\cref{sec:computational_results} reports the computational results.
Finally, \cref{sec:Conclusions} draws the conclusion.

Throughout the paper, for any $n \in \mathbb{Z}_+$, let $[n] = \{1, 2, \dots, n\}$, and $[n] = \varnothing$ if $n = 0$.  
{Let $\boldsymbol{0}$, $\bm{1}$, and $\be_j$ be the all-zeros, all-ones, and the $j$-th unit vectors, respectively, of appropriate dimension.}
For any ${x} \in \{0, 1\}^n$, we define its support as $\mathcal{S}^x = \{j \in [n] : x_j = 1\}$.  
For any subset $\CS \subseteq [n]$, we denote the cardinality and the collection of all subsets of $\CS$  by $|\CS|$ and $2^{\CS}$, respectively.
For any real number $a$, let $(a)^+ = \max\{a, 0\}$.

\section{Problem formulation}\label{sect:problem-formulation}

{We first present the problem formulation for the \SCFLP under partially binary rule \citep{Biesinger2016}. }
In this problem, two firms, a leader and a follower, compete for the customer demand in a shared market.	 
Let $[m]$ be the set of customers and $[n]$ be the set of potential facility locations considered by the leader and follower (for simplicity of discussion, we assume that the potential facility locations, considered by the leader and the follower, are identical;
the extension to the case that  the leader and the follower consider different potential facility locations is straightforward).
{Under partially binary rule \citep{Hakimi1990},}
each customer $i \in [m]$ splits all of his/her demand {$w_i>0$} over the nearest leader facility $j \in [n]$ and the nearest follower facility $\ell \in [n]$ 
proportional to their attractiveness {$v_{ij}>0$} and {$v_{i\ell}>0$}, respectively.
The attractiveness parameters $v_{ij}$ and $v_{i\ell}$ are predetermined by a set of attributes 
such as distance, transportation cost, waiting time, and facility capacity.
In the \SCFLP,
the leader attempts to open $p$ $(p \geq 1)$ facilities to maximize his/her market share 
anticipating that the follower will react to the decision by opening $r$ $(r \geq 1)$ facilities.

Let $x,  \, y \in  \{0, 1\}^n$ be the leader/follower location variables such that $x_j = 1$/$y_j = 1$ 
if the leader/follower deploys a facility at location $j$ and $x_j = 0$/$y_j = 0$ otherwise.
Mathematically, the probability of customer $i$ (or the proportion of his/her demand)
patronizing the leader's facilities can be presented as
\begin{equation}
	h_i(x,y): = \frac{\max_{j \in [n]}v_{ij}x_j}{\max_{j \in [n]}v_{ij}x_j + {\max_{j \in [n]}v_{ij}y_j}}, 
\end{equation}
and the leader's market share is given by 
\begin{equation}\label{def:Gxy}
	g(x,y) := \sum_{i=1}^m w_i h_i(x,y).
\end{equation} 
The \SCFLP under partially binary rule can be formulated as the following \MINLP  problem:
\begin{subequations}\label{origin-bi-level}
	\begin{align}
	\text{(SCFLP)} \quad	\max \ &  \sum_{i=1}^m w_i \frac{\max_{j \in [n]}v_{ij}x_j}{\max_{j \in [n]}v_{ij}x_j + {\max_{j \in [n]}v_{ij}y^*_j}} \label{up-obj}\\
		\St \ & \sum_{j=1}^nx_j = p, \label{up-1}\\
		&x_j \in \{0,1\}, \ \forall \ j \in [n], \label{up-3}
	\end{align}
\end{subequations}
where
\begin{subequations}\label{origin-bi-level-lower}
	\begin{align}
		y^\ast \in \argmax\  &\sum_{i=1}^m w_i \frac{\max_{j \in [n]}v_{ij}y_j}{\max_{j \in [n]}v_{ij}x_j + {\max_{j \in [n]}v_{ij}y_j}}\label{low-obj}\\
		\St\ & \sum_{j=1}^ny_j = r, \label{low-1}\\
		&y_j \in \{0,1\},\ \forall \ j \in [n]. \label{low-3}
	\end{align}
\end{subequations}
The objective functions \eqref{up-obj} and \eqref{low-obj}  maximize 
the market share of the leader and the follower, respectively. 
Constraints \eqref{up-1} and \eqref{low-1} ensure that the leader and follower open exactly $p$ and $r$ facilities, respectively.
Finally, constraints \eqref{up-3} and \eqref{low-3} restrict the decision variables to be binary integers.

Next, we follow \cite{Qi2024} to rewrite problem \eqref{origin-bi-level} as a single-level \MINLP problem.
To proceed, observe that the objective functions \eqref{up-obj} and \eqref{low-obj} of the upper-level and lower-level problems \eqref{origin-bi-level} and \eqref{origin-bi-level-lower}, respectively, add up to a constant (i.e., $\sum_{i=1}^m w_i$).
Intuitively, this is determined by the fact that a customer
splits all of his/her demand over the leader and the follower.
This observation implies that the lower-level problem \eqref{origin-bi-level-lower} is equivalent to
\begin{equation}\label{def:Y}
	y^\ast \in \argmin_{y \in \Y} g(x,y), \text{ where } \Y:= \{y:\eqref{low-1},\ \eqref{low-3}\},
\end{equation}
where $g(x,y)$ is defined in \eqref{def:Gxy}.
Thus, the bilevel programming problem \eqref{origin-bi-level} of the \SCFLP can be rewritten as the following {max-min} problem: 
\begin{equation}\label{def:X}%\label{bi-level-P2}
	\max_{x \in \X} \min_{y \in \Y}g(x,y), \text{ where } 	\X:= \{x: \eqref{up-1}, \ \eqref{up-3}\}.
\end{equation}
By introducing a continuous variable $\eta$ to capture the optimal value of the lower-level problem, problem 
 \eqref{origin-bi-level} can be written as the following \MINLP formulation: 
\begin{align}\label{single-level-P1}\tag{\MINLP}
		\max_{\eta, \, x \in \X} \left\{\eta: \eta \le g(x,y), \ \forall \ y \in \Y \right\}.
\end{align} 

Unfortunately, \eqref{single-level-P1}
has an exponential number of {nonlinear constraints} $\eta\leq g(x,y)$ (as the number of points in $\Y$ could be exponential) involving the highly nonlinear multiple ratio functions $g(x,y)$, which
makes it unrealistic to find an optimal solution of the \SCFLP by directly solving \eqref{single-level-P1}.
In order to bypass the difficulty, in the next three sections, we will investigate the polyhedral structure of the mixed 0-1 set corresponding to each individual nonlinear constraint:
\begin{equation}\label{def:U_y_max}
	\CU_y = \left\{(\eta, x) \in \R \times \{0,1\}^n: \eta \le g(x,y)\right\},
\end{equation}%
and develop equivalent \MILP reformulations of  \eqref{single-level-P1} that can be solved to optimality by state-of-the-art \LP-based \BnC algorithmic framework.

\section{An \MILP reformulation based on submodular inequalities}
\label{sect:milp_sub}

In this section, we follow \cite{Qi2024} to derive an \MILP formulation for the \SCFLP under partially binary rule using submodular inequalities. 
We start from the definition of submodular function  and the linear description of a submodular set \citep{Nemhauser1981}. 
\begin{definition}%[submodular functions]
	A set function $F$: $2^{[n]} \rightarrow \R$ is nondecreasing and submodular if and only if
	\begin{equation*}
		F(\CS \cup \{j\}) - F(\CS) \ge F(\CT \cup \{j\}) - F(\CT) \geq 0, ~\forall~\CS \subseteq \CT \subseteq [n],~\forall~j \in [n]\backslash \CT.
	\end{equation*}
\end{definition}
\noindent Intuitively, if the marginal gain $\rho_j^{F}(\CS) := F(\CS \cup \{j\}) - F(\CS)$ 
of adding an element $j$ to set $\CS \subseteq [n]$ diminishes with the size of $\CS$,
then function $F(\CS)$ is submodular.
For a nondecreasing and submodular function $F$, {let $f$: $\{0,1\}^n \rightarrow \R$ be such that $f(x)= F(\CS^x)$  where $x \in \{0,1\}^n$ and $\CS^x = \{j \in [n] : x_j = 1\}$,} and let 
\begin{equation}\label{def:general U}
	\U = \left\{(\eta, x) \in \R \times
	\{0,1\}^n: \eta \le f(x) \right\}
\end{equation}
be a submodular set.
Using the classic result in \citet{Nemhauser1981}, we can provide a linear description of $\U$.
\begin{theorem}[\citet{Nemhauser1981}]\label{thm:sub-Nemhauser1981}
	$\U = \{ (\eta, x) \in \R \times \{0,1\}^n: \eqref{ineq:submod}\}$,
	where 
	\begin{equation}\label{ineq:submod} 
		\eta \le F(\CS) + \sum_{j \in [n] \backslash \CS} \rho_j^{F}(\CS) x_j,
		\ \forall \ \CS \subseteq [n].
	\end{equation}
\end{theorem}

Given $y\in \Y$, let $H_{i,y} : 2^{[n]}\rightarrow \mathbb{R}_+$ and $G_{y} : 2^{[n]}\rightarrow \mathbb{R}_+ $ be the set functions defined by 
\begin{equation}\label{def:CH,CG}
	H_{i,y}(\CS) =  
	\frac{ \max_{j \in \CS}v_{ij} }{\max_{j \in \CS} v_{ij} + \max_{k \in [n]} v_{ik}y_k}, ~ \forall~ i \in [m],~\text{and}~{G_{y}(\CS)= \sum_{i=1}^mw_i H_{i,y}(\CS)},~\forall~\CS \subseteq[n].
\end{equation} 
By definitions, if the leader and follower open the facilities induced by $x \in \{0,1\}^n$ and $y \in \{0,1\}^n$, i.e., those in $\mathcal{S}^x = \{j \in [n] : x_j = 1\}$ and $\mathcal{S}^y = \{j \in [n] : y_j = 1\}$, then $H_{i,y}(\CS^x) $ is the probability of customer $i$  patronizing the leader's facilities $\CS^x$, and $G_{y}(\CS^x)$ is the market share captured by the leader.
As such, it follows $H_{i,y}(\CS^x) = h_i(x,y)$, $i \in [m]$, and $G_{y}(\CS^x) = g(x,y)$ for $x,\, y \in \{0,1\}^n$. 
Due to this, $H_{i,y}(\CS^x)$ and $h_i(x,y)$ (respectively, $G_{y}(\CS^x) $ and $g(x,y)$) are used interchangeably in the subsequent discussions.

\begin{proposition}\label{thm:submodularity}
	Given $y \in \Y$, let 
	\begin{equation}\label{def:c^y_{ij}}
		c^y_{ij} := \frac{v_{ij}} 
		{v_{ij} + \max_{k \in [n]}v_{ik}y_k}, \ \forall \  i \in [m],\ \forall \   j \in [n].
	\end{equation}
	Then, the following statements hold:
	(i) for $i \in [m]$ and $\CS \subseteq  [n]$, $H_{i,y}(\CS)=\max_{j \in \CS} c^y_{ij}$;
	(ii) for	 $i \in [m]$, $H_{i,y}$ is nondecreasing and submodular;
	(iii) $G_{y}$ is nondecreasing and submodular. 
\end{proposition}
\begin{proof}{Proof}
	Statement (i) holds since 
	\begin{equation}\label{tmpeq1}
		H_{i,y}(\CS) =  
		\frac{ \max_{j \in \CS}v_{ij} }{\max_{j \in \CS} v_{ij} + \max_{k \in [n]} v_{ik}y_k} 	\overset{(a)}= \max_{j \in \CS}\frac{ v_{ij} }{v_{ij} + \max_{k \in [n]} v_{ik}y_k}=\max_{j \in \CS} c^y_{ij},
	\end{equation}
	where (a) follows from the reasons that $\max_{k \in [n]} v_{ik}y_k$ is a constant and the maximum of $\{\frac{ v_{ij} }{v_{ij} + \max_{k \in [n]} v_{ik}y_k}\}_{j \in \CS}$ is attained at $j_0 \in \argmax_{j \in \CS}v_{ij}$.
	Statement (ii) follows from \eqref{tmpeq1} and the fact that for $c\in \mathbb{R}^{n}_+$, the max function $M(\CS)=\max_{j\in \CS} c_j $ is nondecreasing and submodular \citep{Nemhauser1981}.
	Statement (iii) follows from the fact that any nonnegative linear  combination of nondecreasing and submodular functions is also nondecreasing and submodular. 
\end{proof}

By \cref{thm:submodularity} (i), we obtain $g(x,y)= \sum_{i=1}^m w_ih_i(x,y)= \sum_{i=1}^mw_i \max_{j \in [n]} c^y_{ij}x_j$ and thus can rewrite the mixed 0-1 set $\CU_y$ in \eqref{def:U_y_max} as
\begin{equation}\label{def:U_y_maxform}
	\CU_y = \left\{(\eta, x) \in \R \times \{0,1\}^n: \eta \le g(x,y)=\sum_{i=1}^mw_i \max_{j \in [n]} c^y_{ij}x_j \right\}.
\end{equation}%
Moreover, using the monotonicity and submodularity of {$G_{y}$} in \cref{thm:submodularity} (iii) 
and the results in \cref{thm:sub-Nemhauser1981},
we can characterize $\U_y$ using the following submodular inequalities
\begin{equation}\label{ineq:submod-G}
	\eta \le G_{y}(\CS) + \sum_{j \in [n] \backslash \CS} \rho_j^{G_{y}}(\CS) x_j, \ \forall \ \CS \subseteq [n].
\end{equation}
That is, $\U_y = \{(\eta, x) \in \R \times \{0,1\}^n: \eqref{ineq:submod-G}\}$. 
Consequently, we can substitute the nonlinear constraints $\eta \leq g(x,y)$ in \eqref{single-level-P1} with the submodular inequalities in \eqref{ineq:submod-G}, and obtain an equivalent \MILP formulation for the \SCFLP:
\begin{align}
	\label{prob:submod-aggr}
	\tag{SF}
	\max_{\eta, \, x \in \X} \left\{ \eta: \eta \le G_{y}(\CS) + \sum_{j \in [n] \backslash \CS} \rho_j^{G_{y}}(\CS) x_j,
	\  \forall \ \CS \subseteq [n], \ \forall \ y \in \Y\right\}.
\end{align}

\section{An improved MILP formulation based on improved submodular  inequalities}\label{section:new-submodular}

In this section, we  shall develop a broader family of valid inequalities for $\CU_y$, which include the submodular inequalities as special cases.
We show that together with the trivial inequities $\bm{0} \leq x \leq \bm{1}$, the improved submodular inequalities can  characterize the convex hull of $\CU_y$. 
This enables to derive a new \MILP formulation for the \SCFLP with a much stronger \LP relaxation.

\subsection{The proposed \MILP formulation}\label{subsec:GSF}
As stated in \cref{thm:sub-Nemhauser1981}, we can use the submodular inequalities \eqref{ineq:submod} to provide a linear characterization of a submodular set $\U$ associated with the nondecreasing and submodular function $F$. 
The following result shows that the linear characterization can be potentially improved in terms of providing a tighter linear relaxation, under the condition that $F$ is a nonnegative combination of nondecreasing and submodular functions $F_1, \ldots, F_m$.

\begin{theorem}\label{thm:decompose-Uy}
Let {$w_1, \ldots, w_m \in  \mathbb{R}_+$} and $F_1, \ldots, F_m: 2^{[n]} \rightarrow \R$ be {a sequence} of nondecreasing and submodular functions, and {for $i \in [m]$, let $f_i : \{0,1\}^n \rightarrow \mathbb{R}$ be such that $f_i(x) = F_i(\CS^x)$ where $x \in \{0,1\}^n$ and  $\mathcal{S}^x = \{j \in [n] : x_j = 1\}$
}.
Define 
\begin{equation*}
\U = \left\{(\eta, x) \in \R \times
\{0,1\}^n: \eta \le \sum_{i =1}^m w_i f_i(x) \right\}.
\end{equation*}
Then $\U = \{ (\eta, x) \in \R \times \{0,1\}^n: \eqref{ineq:thm-decompose Uy}\}$, where 
\begin{equation}\label{ineq:thm-decompose Uy}
\eta \le \sum_{i = 1}^m w_i\left(F_{i}(\CS_i) + \sum_{j \in [n] \backslash \CS_i} \rho_j^{F_{i}}(\CS_i) x_j\right), 
\ \forall \ (\CS_1, \dots, \CS_m) \in 2^{[n]} \times \cdots \times 2^{[n]}.
\end{equation}
\end{theorem}

\begin{proof}{Proof}
Let 
\begin{equation}\label{set:Uy-rewrite-separable}
\CZ = \left\{(\eta,\zeta, x) \in \R \times \R^{m} \times\{0,1\}^n: ~ \eta = \sum_{i =1}^m w_i\zeta_i,~
(\zeta_i, x) \in \CZ_{i}, \ \forall \  i \in [m]\right\},
\end{equation}
where $\CZ_{i} := \left\{ (\zeta_i, x) \in \R \times \{0,1\}^n: \zeta_i \le  f_i(x)\right\}$ for  $i \in [m]$.
Then $\proj_{\eta,x} (\CZ) =  \U$.
For each $i \in [m]$, as $F_{i}$ is nondecreasing and submodular, from the results in \cref{thm:sub-Nemhauser1981}, 
$\CZ_i$ can be equivalently characterized by submodular inequalities,
i.e., {$\CZ_{i} = \{(\zeta_i, x)\in \R 
\times \{0,1\}^n: \eqref{ineq:submod-Hi}\}$}, where
\begin{equation}\label{ineq:submod-Hi}
\zeta_i \le F_{i}(\CS_i) + \sum_{j \in [n] \backslash \CS_i} \rho_j^{F_{i}}(\CS_i) x_j,
\ \forall \ \CS_i \in 2^{[n]}. 
\end{equation}
Thus, $\CZ = \{(\eta,\zeta,x) \in \R \times \R^m \times \{0,1\}^n: \eta = \sum_{i =1}^m w_i\zeta_i,~\eqref{ineq:submod-Hi}\}$. 
Projecting variables $\zeta$ out from $\CZ$, we obtain 
$\U = \{(\eta, x) \in \R \times \{0,1\}^n: \eqref{ineq:thm-decompose Uy}\}$. 
\end{proof}

For each $y \in \Y$, function {$G_{y}$} is a nonnegative
combination of nondecreasing and submodular functions ${H_{i,y}}$, $i \in [m]$.
Thus, it follows from \cref{thm:decompose-Uy} that $\U_y = \{(\eta, x) \in \R \times \{0,1\}^n: \eqref{ineq:submod-G-D}\}$ where 
\begin{equation}\label{ineq:submod-G-D}
\eta \le \sum_{i=1}^m w_i \left(H_{i,y}(\CS_i) + \sum_{j \in [n] \backslash \CS_i} \rho_j^{H_{i,y}}(\CS_i) x_j\right), 
\ \forall \ (\CS_1, \dots, \CS_m) \in 2^{[n]} \times \cdots \times 2^{[n]}.
\end{equation}
Consequently, {we can substitute $\eta \leq g(x,y)$ in \eqref{single-level-P1} with the improved submodular inequalities in \eqref{ineq:submod-G-D} 
 and obtain the following equivalent \MILP formulation:}
{\small 
\begin{equation}\label{prob:submod-disaggr} 
\begin{aligned}
\max_{\eta, \, x \in \X} 
& \left\{
\eta: \eta \le \sum_{i=1}^m w_i  \left(H_{i,y}(\CS_i) + \sum_{j \in [n] \backslash \CS_i} 
\rho_j^{H_{i,y}}(\CS_i) x_j\right), 
	 \forall \ (\CS_1, \dots, \CS_m) \in 2^{[n]} \times \cdots \times 2^{[n]},
~\forall~y \in \Y \right\}.
\end{aligned}
\end{equation}
 }%
Observe that by setting $\CS_1= \cdots =  \CS_m= \CS$, inequalities \eqref{ineq:submod-G-D} reduce to the inequalities \eqref{ineq:submod-G}, 
and thus the improved submodular  inequalities in \eqref{ineq:submod-G-D} include the submodular inequalities in \eqref{ineq:submod-G}.
As a result, the \LP relaxation of problem \eqref{prob:submod-disaggr} is at least as tight as {that of problem} \eqref{prob:submod-aggr}. 
In Appendix \ref{appendix1}, we provide an example to illustrate that the \LP relaxation of problem  \eqref{prob:submod-disaggr}
can be strictly stronger than that of problem \eqref{prob:submod-aggr}.

\subsection{Strength of  formulation \eqref{prob:submod-disaggr}}
Next, we characterize the  strength of  formulation \eqref{prob:submod-disaggr}. 
To achieve this, we need the following result.
\begin{lemma}\label{thm:sub-Nemhauser1981-max}
For $c \in \mathbb{R}^n_+$, let $f(x) =\max_{j \in [n]} c_j x_j $,
{$F$: $2^{[n]} \rightarrow \R_+$ be such that $F(\CS^x)= f(x)$  where $x \in \{0,1\}^n$ and $\mathcal{S}^x = \{j \in [n] : x_j = 1\}$,}
and $\U$ be defined in \eqref{def:general U}. 
Then $F$ is a nondecreasing and submodular function 
and the submodular inequalities in \eqref{ineq:submod} reduce to 
\begin{equation}\label{ineq:submod-explicit}
\eta \leq c_{\ell} +  \sum_{j=1}^{n} \left( c_j -
c_{\ell}\right)^+ x_j, \ \forall \ {\ell \in [n+1]},
\end{equation}
where $c_{n+1}=0$ and ${(\cdot)}^{+} = \max\{0, \cdot\}$.
Moreover, $\conv(\CU) 
= \{(\eta, x) \in \R \times [0,1]^n: \eqref{ineq:submod-explicit}\}$. 
\end{lemma}
\begin{proof}{Proof}
	The first part can be obtained from Theorem 8 in \citet{Nemhauser1981}  and the second part can be obtained from Theorem 3.8 in \citet{Chen2024}. 
\end{proof}

Combining  \cref{thm:decompose-Uy}, \cref{thm:sub-Nemhauser1981-max}, and $g(x,y)=  \sum_{i=1}^mw_i \max_{j \in [n]} c^y_{ij}x_j$ (as implied by \cref{thm:submodularity}), the submodular inequalities in \eqref{ineq:submod-G-D} can be written as 
\begin{equation}\label{ineq:submod-explicit-G-D}
\eta \leq \sum_{i=1}^m w_i \left( c^y_{i\ell_i} +  \sum_{j=1}^{n} \left( c^y_{ij} -
c^y_{i\ell_i}\right)^+ x_j\right), \ \forall \ (\ell_1, \ldots, \ell_m)  \in [n+1]^m,
\end{equation}
where $c^y_{i(n+1)} = 0$ for all $i \in [m]$ and $y \in\Y$, and  $[n+1]^m = [n+1]\times \dots \times[n+1]$ denotes the Cartesian product.
Thus, formulation 
%\MILP formulation
\eqref{prob:submod-disaggr} reduces to
	\begin{align}	\label{prob:submod-disaggr-card1} \tag{GSF}
		& \max_{\eta, \, x \in \X}   \left\{
		\eta: \eta \leq \sum_{i=1}^m w_i \left( c^y_{i\ell_i} +  \sum_{j=1}^{n} \left( c^y_{ij} -
		c^y_{i\ell_i}\right)^+ x_j\right), \ 	\forall \ (\ell_1, \dots, \ell_m) \in [n+1]^m, \ \forall~ y \in \Y \right\}. %\right.\\ 
	\end{align}

The following theorem further characterizes  
the strength of formulation \eqref{prob:submod-disaggr-card1} (or \eqref{prob:submod-disaggr}) in terms of 
providing a strong \LP relaxation bound. 
The proof is provided in Appendix \ref{sec:appendix-pf-convexhull}.

\begin{theorem}\label{thm:submod-convexhull}
For all $y \in \Y$, 
it follows that
\begin{equation}\label{convU}
\conv(\CU_y) %= \CU_y^{\SC\text{--L}}
=\{(\eta, x) \in \R \times [0,1]^n: \eqref{ineq:submod-explicit-G-D}\}.
\end{equation}
\end{theorem}

\section{A strong extended formulation}\label{section:ZLP}
In Section \ref{section:new-submodular}, we develop an \MILP formulation  \eqref{prob:submod-disaggr-card1},
which is based on a linear characterization of $\conv(\U_y)$ for each $y \in \Y$. 
Each of these characterizations, however, uses an exponential number of improved submodular inequalities in \eqref{ineq:submod-explicit-G-D}.
In this section, our goal is to develop an alternative formulation for the \SCFLP 
that does not suffer from the exponential number of improved submodular inequalities.
In particular, we first derive a compact extended formulation for $\conv(\U_y)$ and  develop an extended \MILP formulation whose \LP relaxation is as tight as that of \eqref{prob:submod-disaggr-card1}.
Then we present a preprocessing technique to eliminate the exponential number of variables in the extended formulation {(due to the exponential number of points in $\Y$)}, which paves the way for solving the extended formulation by a \BnC algorithm.

\subsection{A compact extended formulation for $\conv(\CU_{y})$}\label{subsection:4.1} %for conv(\CU_{y}) 
{Given $y \in \Y$, to derive an extended formulation for $\conv(\CU_y)$,} let us introduce binary variables  $\{z^y_{ij}\}_{i\in [m],~j \in [n]}$ and 
\begin{equation}\label{def:W_Y}
	\begin{aligned}
			\W_{y}  =  \left\{{(\eta,x,z^y) \in \R \times \{0,1\}^n \times \{0,1\}^{mn}}: 
			\ \eqref{z-linear-3}, \ \eqref{z-linear-4},~\eqref{z-linear-1} \right\},
		\end{aligned}
\end{equation}
where
	\begin{align}%\tag{Z2}
			& z^y_{ij} \leq x_j,\ \forall \  i \in [m], \ \forall \ j \in [n], \label{z-linear-3}\\
			&\sum_{j=1}^nz^y_{ij} \leq 1, \  \forall \  i \in [m],\label{z-linear-4}\\
			&\eta \leq \sum_{i=1}^m w_i \sum_{j=1}^n c^y_{ij} z^y_{ij}. \label{z-linear-1}
	\end{align}
Here, we denote $z_{ij}^y = 1$ if  the customer $i$ patronizes facility $j$  and $z_{ij}^y = 0$ otherwise.
Constraints \eqref{z-linear-3} 
ensure that the customers can only be served by the open facilities.
Constraints 
\eqref{z-linear-4}
guarantee that each customer patronizes at most one open facility.
Together with \eqref{z-linear-3} and \eqref{z-linear-4}, the right-hand side of \eqref{z-linear-1} characterizes the market share captured by the leader {(if the follower's solution is $y$)}.
The following theorem shows that $\W_y$ is an extended formulation of $\CU_y$ defined in \eqref{def:U_y_max}. {The proof  is provided in Appendix \ref{sec:appendix-pf-proj-Wy}.}
\begin{theorem}\label{thm:proj-Wy}
	For $y \in \Y$, $\proj_{\eta, x} (\W_{y}) = \CU_{y}$. 
\end{theorem}

Let
\begin{equation*}
	\W_{y}^{\text{L}} :=  \left\{(\eta,x,z^{y}) \in \R \times  [0,1]^n \times [0,1]^{mn}: 
	\eqref{z-linear-3}, \ \eqref{z-linear-4},\ \eqref{z-linear-1}\right\}
\end{equation*}
be the linear relaxation of $\W_{y}$. 
Next, we further analyze the relations of $\proj_{\eta, x} (\W_{y}^\text{L})$ and $\conv(\U_{y})$. 
Let 
\begin{equation}\label{def:g_i}
	\begin{aligned}
		&g_{i,\ell_i}^{y}(x) := c^{y}_{i\ell_i}+  \sum_{j=1}^{n} \left( c^{y}_{ij}-
		c^{y}_{i\ell_i}\right)^+ x_j, ~\forall  ~\ell_i \in [n+1], ~\forall~ i \in [m], \\ 
		& g_{i}^{y}(x) := \min_{\ell_i \in [n+1]}g_{i,\ell_i}^{y}(x),~\forall ~ i \in [m],
	\end{aligned}
\end{equation}
where we denote that $c_{i(n+1)}^y = 0$ for $i \in [m]$ and $y \in \Y$.
Then, by \cref{thm:submod-convexhull}, we obtain
\begin{equation}\label{def:U_y_SC}
	\small
	\begin{aligned}
		\conv(\CU_{y}) 
		%		&= \CU_{y}^{\SC\text{--L}} \\
		&=
		\{(\eta, x) \in \R \times [0,1]^n: \eqref{ineq:submod-explicit-G-D}\} 
		= \left\{(\eta, x) \in \R \times [0,1]^n: \eta \le \min_{(\ell_1, \ldots, \ell_m) 
			\in [n+1]^m}{\sum_{i =1}^m} w_i 		g_{i,\ell_i}^{y}(x) \right\}\\
		&\overset{(a)}{=} \left\{(\eta, x) \in \R \times [0,1]^n: \eta \le \sum_{i =1}^m w_i \min_{\ell_i \in [n+1]}	
		g_{i,\ell_i}^{y}(x) \right\}
		= \left\{(\eta, x) \in \R \times [0,1]^n: \eta \le \sum_{i =1}^m w_i g_{i}^{y}(x) \right\},
	\end{aligned}
\end{equation}
{where {(a) follows from $w_i \geq 0$ and $g_{i,\ell_i}^{y}(x)\geq 0$ for all $i \in [m]$}.}
On the other hand, letting 
\begin{equation}\label{def:h_i}
	f_{i}^{y}(x) := \max_{z^{y}_{i\cdot}} \left\{\sum_{j=1}^nc^{y}_{ij}z^{y}_{ij} : 0 \le z_{ij}^{y} \leq x_j, 
	\ \forall \ j \in [n], \ \sum_{j=1}^nz^y_{ij} \leq 1 \right\},~\forall~i \in [m],
\end{equation}
we have 
\begin{equation}\label{projW}
	\proj_{\eta, x} (\W_{y}^{\text{L}}) = \left\{ (\eta,x) \in \mathbb{R}\times [0,1]^n \, : \, \eta\leq \sum_{i=1}^m w_i f_{i}^{y}(x)  \right \}.
\end{equation}
To derive a closed formula for $f_{i}^{y}(x)$, let us consider the dual of the \LP problem \eqref{def:h_i}:
\begin{equation}\label{def:h_i dual}
	f_{i}^{y}(x) = \min_{u_i, \, w_{i\cdot}} \left\{u_i + \sum_{j=1}^nx_{j} w_{ij} : u_i + w_{ij} \geq c^{y}_{ij},
	\ \forall \ j \in [n], \ u_i  \geq 0,\  w_{ij} \geq 0, \ \forall \ j \in [n]  \right\}.
\end{equation}
By $x_j \geq 0$, there must exist an optimal solution $(u_i, w_{i\cdot})$ of problem \eqref{def:h_i dual} such that   $w_{ij}= (c^{y}_{ij}- u_i)^+$ holds for all $j \in [n]$, and thus  
\begin{equation}\label{tempdualopt}
	f_{i}^{y}(x) = \min_{u_i \in \R_{+}} \left(u_i +\sum_{j=1}^n  x_j(c^{y}_{ij}- u_i)^+ \right).
\end{equation}
The objective function of problem \eqref{tempdualopt} is a piecewise linear function that is bounded below on $\mathbb{R}_+$ and changes slope whenever $u_i = c^y_{ij}$ for some $j \in [n+1]$. 
Hence, its minimum is obtained at some $u_i = c^y_{i\ell_i}$  with $\ell_i \in [n+1]$, which indicates  
\begin{equation}\label{hdef}
	f_{i}^{y}(x) = \min_{\ell_i \in [n+1]} \left(c^{y}_{i\ell_i}+ \sum_{j=1}^n (c^{y}_{ij}- c^{y}_{i\ell_i} )^+x_{j} \right) 
	= g^y_i(x).
\end{equation}
Combining \eqref{def:U_y_SC}, \eqref{projW}, and \eqref{hdef}, we obtain the equivalence of $\conv(\U_{y})$ and  $\proj_{\eta, x} (\W_{y}^{\text{L}}) $.
\begin{theorem}\label{thm:proj_Wy^C}
		Given $y \in \Y$,  $\proj_{\eta, x} (\W_{y}^{\text{L}})   = \conv(\U_{y}) $.
\end{theorem}
\cref{thm:proj_Wy^C} demonstrates that $\conv(\U_y)$ enjoys an extended formulation $\W_y^\text{L}$ in the $(\eta, x, z)$ space with only  $mn + m + 1$ inequalities. 
{This is different from the linear characterization in the $(\eta , x)$ space where an exponential number of improved submodular inequalities  \eqref{ineq:submod-explicit-G-D} is needed.}

\subsection{An extended formulation for  \SCFLP}
\label{subsection:ZLP}

Based on the results in \cref{thm:proj_Wy^C}, we immediately obtain an extended \MILP formulation
of problem \eqref{single-level-P1}: 
	\begin{equation}\label{prob:ZLP-pre}
			\max_{\eta,\, x \in \X,\, \{z^y\}_{y \in \Y} } \left\{\eta:  
 \eqref{z-linear-3}, \ \eqref{z-linear-4}, \ \eqref{z-linear-1},
	~z^y \in \{0,1\}^{mn},~\forall \  y  \in \Y\right\}.
	\end{equation}
\noindent Although \eqref{prob:ZLP-pre} is an \MILP problem, its exponential number
of allocation variables $z^y_{ij}$ 
makes it unrealistic to be directly solved by a \BnC algorithmic framework. 
{The following theorem shows that the exponential family of variables $\{z^y\}_{y \in \Y}$ (as $|\Y|$ is exponential) in \eqref{prob:ZLP-pre}, however, can be aggregated into $mn$ variables.
The proof is provided in Appendix \ref{sec:appendix-pf-exist-z^y}.}
\begin{theorem}\label{thm:exist-z^y}
	For problem \eqref{prob:ZLP-pre} (respectively, its \LP relaxation), {there exists an optimal solution $(\eta, x, \{{z}^y\}_{y \in \Y})$ and $\hat{z} \in \{0,1\}^{mn}$ (respectively, $\hat{z} \in [0,1]^{mn}$)
	 such that $z^y = \hat{z}$ for all $y \in \Y$.}
\end{theorem}

Using \cref{thm:exist-z^y}, we can introduce new binary variables $z \in \{0,1\}^{mn}$ in place of all binary variables $\{z^y\}_{y \in \Y}$ from \eqref{prob:ZLP-pre}, and obtain the following equivalent problem formulation
	\begin{align}\label{prob:ZLP}\tag{EF}
		\max_{\substack{\eta,\, x \in \X,\\ z \in \{0,1\}^{mn}}}  
		&\left\{\eta : 
		\eqref{z-linear-3-no-y}, ~ 
		\eqref{z-linear-4-no-y}, ~
		\eqref{z-linear-1-no-y}\right\},
	\end{align} 
where 
\begin{align}%\tag{Z2}
	& z_{ij} \leq x_j,\ \forall \  i \in [m], \ \forall \ j \in [n], \label{z-linear-3-no-y}\\
	&\sum_{j=1}^n z_{ij} \leq 1, \  \forall \  i \in [m],\label{z-linear-4-no-y}\\
	&\eta \leq \sum_{i=1}^m w_i \sum_{j=1}^n c^y_{ij} z_{ij},\ \forall \ y \in \Y. \label{z-linear-1-no-y}
\end{align}
Compared with formulation \eqref{prob:ZLP-pre} which involves a potentially exponential number of variables $\{z^y\}_{y \in \Y}$, the new formulation \eqref{prob:ZLP} only involves $mn$ $z$ variables, which makes \eqref{prob:ZLP} realistic to be solved by a \BnC algorithmic framework.

{Two remarks on formulation  \eqref{prob:ZLP} are in order. }
First, 
by \cref{thm:exist-z^y}, the \LP relaxation of \eqref{prob:ZLP-pre} is equivalent to that of \eqref{prob:ZLP}, which, along with \cref{thm:proj_Wy^C},
implies the equivalence of the \LP relaxations of \eqref{prob:ZLP} and \eqref{prob:submod-disaggr-card1}.
\begin{corollary}\label{coro:ZLP=original_MINLP}
	The \LP relaxation of problem \eqref{prob:ZLP} is equivalent to that of problem \eqref{prob:submod-disaggr-card1}
	in terms of sharing the same optimal value.
\end{corollary} 
\noindent Second, 
{unlike \eqref{prob:submod-disaggr-card1} which involves \((n+1)^m\) improved submodular inequalities \eqref{ineq:submod-explicit-G-D} for each \(y \in \Y\),
formulation \eqref{prob:ZLP} requires only a single inequality \eqref{z-linear-1-no-y}  for each \(y \in \Y\) (although this comes at a cost of including the $mn$ $z$ variables).}
This feature makes formulation \eqref{prob:ZLP} able to serve as a good alternative to be embedded in a \BnC algorithmic framework for finding a global optimal solution to the \SCFLP.

\section{The overall algorithmic framework}\label{section: implementation}
In this section, we present the overall algorithmic framework for solving the three \MILP reformulations \eqref{prob:submod-aggr},
 \eqref{prob:submod-disaggr-card1}, and \eqref{prob:ZLP} of the \SCFLP, developed in 
 \cref{sect:milp_sub,section:new-submodular,section:ZLP}. 
Note that due to the exponential families of inequalities \eqref{ineq:submod-G}, \eqref{ineq:submod-explicit-G-D}, and  \eqref{z-linear-1-no-y}, it is impractical to solve \eqref{prob:submod-aggr},
\eqref{prob:submod-disaggr-card1}, and \eqref{prob:ZLP} by directly invoking off-the-shelf \MILP solvers. 
Therefore, we apply a \BnC framework (detailed in Appendix \ref{sec:appendix-BnCframework}) to solve \eqref{prob:submod-aggr},
\eqref{prob:submod-disaggr-card1}, and \eqref{prob:ZLP}  where {the exponential families of inequalities} \eqref{ineq:submod-G}, \eqref{ineq:submod-explicit-G-D}, and  \eqref{z-linear-1-no-y} are separated on the fly. 
In the following, we will detail the separation algorithms to find violated inequalities \eqref{ineq:submod-G}, \eqref{ineq:submod-explicit-G-D}, and  \eqref{z-linear-1-no-y} for the solutions encountered at the nodes of the search tree.

\subsection{Separation  of submodular inequalities \eqref{ineq:submod-G}}\label{subsec:sepa_newsub}
Given a point $(\eta^*, x^*) \in \R \times [0,1]^n $, 
the separation problem of 
the submodular inequalities \eqref{ineq:submod-G}
asks to find inequalities violated by $(\eta^*, x^*)$ or prove none exists. 
This is equivalent to solving the following optimization problem: 
\begin{equation}\label{sepa:oldsub}
		\min_{y \in \Y} \min_{\CS\subseteq [n]} \left(G_{y}(\CS) + \sum_{j \in [n] \backslash \CS} \rho_j^{G_{y}}(\CS) x^*_j\right). 
\end{equation}
To solve \eqref{sepa:oldsub}, we first consider the case $x^*\in \{0,1\}^n$.
Note that for a fixed $y$, the inner problem of \eqref{sepa:oldsub} is
\begin{equation}\label{prob-mini-smod-old}
	\min_{\CS\subseteq [n]} \left(G_{y}(\CS) + \sum_{j \in [n] \backslash \CS} \rho_j^{G_{y}}(\CS) x^*_j\right),
\end{equation}   
and by {monotonicity and submodularity}  of $G_{y}$, one of its optimal solutions is $\CS^{x^*} = \{j \in [n]: x^*_j = 1\}$.
Thus, problem \eqref{sepa:oldsub} reduces to 
 \begin{equation}\label{pmp-oldsub-layer2}
	\begin{aligned}
		\min_{y \in \Y}  \min_{\CS\subseteq [n]} \left(G_{y}(\CS) + \sum_{j \in [n] \backslash \CS} \rho_j^{G_{y}}(\CS) x^*_j\right) 
		= \min_{y \in \Y}G_{y}(\CS^{x^*}) 
		{\overset{(a)}{=}}\min_{y \in \Y}g(x^*,y),
	\end{aligned}
\end{equation}
where {(a)} follows from
 the definitions of $g(x^*, y)$
 and $G_{y}(\CS^{x^*})$ in \eqref{def:Gxy} and \eqref{def:CH,CG}, respectively.
Letting $c_i := \max_{j \in [n]} v_{ij} x^*_j$ for $i \in [m]$ and $a_{ik}:= \frac{c_i}{c_i+ v_{ik}}$ for $i\in [m]$ and $k \in [n]$, 
then for any $y \in \Y \subseteq \{0,1\}^n$, it follows
{\small
 \begin{equation*}
 	\begin{aligned}
 	h_{i}(x^*,y) & = \frac{\max_{j \in [n]} v_{ij} x^ *_j}{\max_{j \in [n]} v_{ij} x^*_j+
 		\max_{k \in [n]} v_{ik} y_k} 
 	& = \frac{c_i}{c_i+ \max_{k \in [n]} v_{ik} y_k} 
 	= \min_{k \in [n]}\frac{c_i}{c_i+ v_{ik}} y_k = \min_{k \in [n]} a_{ik} y_k,
 	\end{aligned}
 \end{equation*}
}%
 and by \eqref{def:Gxy},
 \begin{equation}\label{sepa:pmp-NLP}
 	\begin{aligned}
 		\min_{y \in \Y} g(x^*,y) =\min_{y \in \Y} \sum_{i=1}^m w_i h_i(x^*,y)
 		=  \min_{y \in \Y} \sum_{i=1}^m w_i\min_{k \in [n]} a_{ik} y_k.
 	\end{aligned}	
 \end{equation}
 Problem \eqref{sepa:pmp-NLP} is equivalent to the well-known $r$-median problem which attempts to establish
  a subset of $r$ facilities among the $n$ potential facility locations such that the sum of the costs $\{w_i a_{ik}\}$ between customers and their closest facilities is minimized.
The $r$-median problem  is a well-known NP-hard combinatorial optimization problem \citep{Hakimi1964} but has  
various efficient heuristics and exact solution methods; see \citet{Avella2007,Garcia2011}, and \cite{Duran-Mateluna2023} among many of them.
In our context, we solve this problem by the state-of-the-art Benders decomposition algorithm of \citet{Duran-Mateluna2023}, which guarantees to find an optimal solution of problem \eqref{sepa:pmp-NLP}.

To further speed up the separation procedure, we also apply an efficient heuristic algorithm, as to avoid solving too many separation problems of the form \eqref{sepa:pmp-NLP}. %or branching. 
Specifically, we use $\F$ to track the set of points $y$, 
each of which is an optimal solution of some separation problem \eqref{sepa:pmp-NLP} (encountered in previous iterations).
Then, to quickly find a submodular inequality \eqref{ineq:submod-G} violated by the current solution $(\eta^*,x^*)$,
 we check, for each $y \in \F$, whether 
$\eta^* > G_{y}(\CS^{x^*})$ holds.
If  the above heuristic procedure successfully detects some violated inequalities, then we directly add these inequalities 
to the problem and skip the procedure of solving \eqref{sepa:pmp-NLP}; otherwise, {we solve the separation problem \eqref{sepa:pmp-NLP} exactly.}

For the case $(\eta^*, x^*) \in \R \times [0,1]^n $ with fractional $x^* \notin \{0,1\}^n$, 
it is {generally difficult} to solve  problem  \eqref{sepa:oldsub} exactly, hence a simple rounding heuristic procedure is considered instead. 
{Specifically, we round $\{x_j^*\}$ to their nearest integers to obtain an integral point $\bar{x} \in \{0,1\}^n$.
Then, for each $y \in \F$, we construct a submodular inequality based on $(\eta^*, \bar{x}) $  
and add it into the problem if it is violated by $(\eta^*, x^*)$. }
It is worthwhile remarking that (i) only separating the submodular inequalities at points $(\eta^*, x^*) $ with integral $x^* \in \{0,1\}^n$  can also ensure the convergence and correctness of the algorithm;
  (ii) separating the submodular inequalities at points $(\eta^*, x^*)$ with fractional {$x^* \notin \{0,1\}^n$}
   can strengthen the \LP relaxation more quickly, thereby usually achieving a better overall performance. 

\subsection{Separation  of improved submodular inequalities \eqref{ineq:submod-explicit-G-D}}
\label{subsection:sepa_SF}
Unlike the separation problem of submodular inequalities \eqref{ineq:submod-G}, we can  exactly
solve the separation problem of improved submodular
 inequalities \eqref{ineq:submod-explicit-G-D} through solving an $r$-median problem for arbitrary $(\eta^*, x^*) \in \R \times [0,1]^n $, as detailed below.
 
 For $(\eta^*, x^*) \in \R \times [0,1]^n $, the separation problem of inequalities \eqref{ineq:submod-explicit-G-D}  is equivalent to solving the following optimization problem: 
\begin{equation}\label{sepa:newsub}
	\min_{y \in \Y} \min_{(\ell_1, \dots, \ell_m) \in [n+1]^m} \sum_{i=1}^m w_i \left( c^y_{i\ell_i} 
	+  \sum_{j=1}^{n} \left( c^y_{ij} - 
	c^y_{i\ell_i}\right)^+ x^*_j\right).
\end{equation}
To solve problem \eqref{sepa:newsub}, we provide the following proposition which characterizes the optimal solution and value of the inner problem. 
The proof is provided in Appendix \ref{sec:appendix-newsub-sepa}.
\begin{proposition}\label{thm:newsub-sepa} 
	Given $y \in  \Y$ and $i \in [m]$, 
	let $\sigma_i(1),\ldots$, $\sigma_i(n)$ be a permutation of $[n]$ satisfying 
	$v_{i\sigma_i(1)} \geq \cdots \geq v_{i\sigma_i(n)}$.
	For $x^* \in [0,1]^n$, 
	let
	\begin{equation}\label{def:k_i-star}
		k_i: = \left\{
		\begin{aligned}
			&1, \quad &&\text{if }~x^*_{\sigma_i(1)} = 1;\\ 
			&\max \left\{\ell \in [n] : \sum_{j=1}^\ell x^*_{\sigma_i(j)} < 1 \right\}, \quad && \text{otherwise}.
		\end{aligned}
		\right.
	\end{equation}
	Then, 
	\begin{equation}\label{pmp-newsub-layer0}
		\footnotesize
		\begin{aligned}
			\min_{(\ell_1, \dots, \ell_m)  \in [n+1]^m} \sum_{i=1}^m w_i \left( c^y_{i\ell_i} +  \sum_{j=1}^{n} \left( c^y_{ij} - 
			c^y_{i\ell_i}\right)^+ x^*_j\right)
			&=  \sum_{i=1}^m w_i \left( c^y_{i\sigma_i(k_i+1)} +  \sum_{j=1}^{n} \left( c^y_{ij} - 
			c^y_{i\sigma_i(k_i+1)}\right)^+ x^*_j\right)\\
			& =\sum_{i=1}^m w_i \min_{k \in [n]} b_{ik} y_k,
		\end{aligned}
	\end{equation}
	where
	$b_{ik} := \frac{\left(1-\sum_{j=1}^{k_i}x^*_{\sigma_i(j)}\right)v_{i\sigma_i(k_i+1)}}{v_{i\sigma_i(k_i+1)} + v_{ik}} +
	\sum_{j=1}^{k_i} \frac{ x^*_{\sigma_i(j)}v_{i\sigma_i(j)}}{v_{i\sigma_i(j)} + v_{ik}} \geq 0$ 
	for $i \in [m]$ and $k \in [n]$.
\end{proposition}
\noindent By \cref{thm:newsub-sepa}, solving  the separation problem \eqref{sepa:newsub} is equivalent to solving the following $r$-median problem:
\begin{equation}\label{pmp-newsub-layer1}
	\begin{aligned}
		\min_{y \in \Y} \min_{(\ell_1, \dots, \ell_m)  \in [n+1]^m} \sum_{i=1}^m w_i \left( c^y_{i\ell_i} +  \sum_{j=1}^{n} \left( c^y_{ij} - 
		c^y_{i\ell_i}\right)^+ x^*_j\right)
		=\min_{y \in \Y}\sum_{i=1}^m w_i \min_{k \in [n]} b_{ik} y_k.
	\end{aligned}
\end{equation}

Similar to the separation of submodular inequalities, to avoid solving too many $r$-median problems \eqref{pmp-newsub-layer1}, we apply
an efficient heuristic algorithm to solve problem \eqref{sepa:newsub}.
The heuristic algorithm attempts to find  violated inequalities \eqref{ineq:submod-explicit-G-D} by solving
the inner subproblems of \eqref{sepa:newsub} corresponding to $y \in \F$. 
{Here, $\F$ consists of points, each of which is an optimal solution of some separation problem \eqref{pmp-newsub-layer1} (encountered in previous iterations).}
In particular, by \cref{thm:newsub-sepa}, we can check, for each $y \in \F$,
whether $\eta^* >\sum_{i=1}^m w_i \left( c^y_{i\sigma_i(k_i+1)} +  \sum_{j=1}^{n} \left( c^y_{ij} - 
c^y_{i\sigma_i(k_i+1)}\right)^+ x^*_j\right)$ holds, and if it holds,  the  improved submodular inequality $\eta \leq \sum_{i=1}^m w_i \left( c^y_{i\sigma_i(k_i+1)} +  \sum_{j=1}^{n} \left( c^y_{ij} - 
c^y_{i\sigma_i(k_i+1)}\right)^+ x_j\right)$ is violated by $(\eta^*, x^*) $.
Since the computation of $\{ c_{ij}^y \}_{i \in [m],\,j \in [n]}$ (defined in \eqref{def:c^y_{ij}}) can be conducted in $\CO(mn)$, it follows that this heuristic algorithm can be accomplished in the complexity of $\CO(|\F|mn)$ (provided that the permutations $\{\sigma_{i}(j) \}_{j \in [n]}$ for $i \in [m]$ have been precomputed). 
If the above heuristic procedure successfully detects some violated inequalities, then we directly add these inequalities 
 to the problem and skip the procedure of solving \eqref{pmp-newsub-layer1}; otherwise, we solve the separation problem \eqref{pmp-newsub-layer1} 
 exactly by the Benders decomposition algorithm of \cite{Duran-Mateluna2023}.
 
\subsection{Separation  of inequalities \eqref{z-linear-1-no-y}}
Similar to  submodular inequalities \eqref{ineq:submod-G} and improved submodular inequalities
 \eqref{ineq:submod-explicit-G-D}, solving the separation problem  of inequalities \eqref{z-linear-1-no-y} 
 is equivalent to solving an  $r$-median problem  as well, as detailed below.

Given a point $(\eta^*, x^*, z^*) \in \R \times [0,1]^n \times [0,1]^{mn}$, the separation problem  of inequalities \eqref{z-linear-1-no-y}  is equivalent to solving the following optimization problem: 
 \begin{equation}\label{prob-mini-Z}
 	\min_{y \in \Y} \sum_{i=1}^m w_i \sum_{j=1}^n c^y_{ij} z^*_{ij}.
 \end{equation}
By the definition of 
$\{c_{ij}^y\}_{(i,j)\in [m] \times [n]}$ in \eqref{def:c^y_{ij}}, we have 
{\small
\begin{equation*}
	\begin{aligned}
		\sum_{j=1}^n c^y_{ij} z^*_{ij} = \sum_{j=1}^n\frac{v_{ij}}{v_{ij}+
			\max_{k \in [n]} v_{ik} {y}_k} z^*_{ij}	=\sum_{j=1}^n \left(\min_{k \in [n]}\frac{v_{ij} }{v_{ij}+
			v_{ik} }{y}_k\right)z^*_{ij}
		\overset{(a)}{=} \min_{k \in [n]}\sum_{j=1}^n \frac{z^*_{ij}v_{ij} }{v_{ij}+
			v_{ik} } {y}_k,
	\end{aligned}
\end{equation*}}% 
where  {$(a)$} follows from 
$\argmin_{k \in [n]} \frac{\beta}{\alpha+v_{ik}}y_k = \argmax_{k \in [n]}v_{ik} y_k$  for any $(\alpha,\beta) \in  \R^2_{+}$.
Letting $d_{ik} := \sum_{j=1}^n \frac{z^*_{ij}v_{ij} }{v_{ij}+v_{ik} }$ for $i \in [m]$ and $k \in [n]$,
problem \eqref{prob-mini-Z} reduces to the following $r$-median problem:
\begin{equation}\label{pmp-zlp}
	\begin{aligned}
		\min_{y \in \Y}\sum_{i=1}^m w_i \min_{k \in [n]} d_{ik} y_k.
	\end{aligned}
\end{equation}
Unlike the separation of inequalities \eqref{ineq:submod-G} and \eqref{ineq:submod-explicit-G-D},
there is no such inner subproblem of problem \eqref{prob-mini-Z}, and we do not design a heuristic  procedure to separate inequalities \eqref{z-linear-1-no-y}.

\section{Computational results}\label{sec:computational_results}
	\vspace*{0.2cm}
	In this section, we first present the computational results to demonstrate the effectiveness of the
	proposed \BnC algorithms for solving the \SCFLP. 
	To do this, we first perform computational experiments to compare the performance of the proposed \BnC algorithms with the state-of-the-art algorithms in \cite{Qi2024} and \cite{Biesinger2016}.  
 	Then, we compare the performance of the proposed \BnC algorithms  based the formulations \eqref{prob:submod-disaggr-card1} 
 and \eqref{prob:ZLP} for solving large-scale \SCFLP{s}. 
	
	The proposed \BnC algorithms were implemented in Julia 1.11.5 using Cplex 20.1.0. 
	Specifically, Cplex was configured to run the code in a single-threaded mode, with a time limit of 7200 seconds and a relative MIP gap tolerance of 0\%. 
	Unless otherwise stated, all other parameters of Cplex were set to their default values. 
	 All computation was conducted on a cluster of Intel(R) Xeon(R) Gold 6230R CPU @ 2.10 GHz computers.

\subsection{Comparison of of the proposed \BnC algorithms with  state-of-the-art algorithms}

We first compare the performance of the \BnC algorithms
based on formulations \eqref{prob:submod-aggr}, \eqref{prob:submod-disaggr-card1},
and \eqref{prob:ZLP} (denoted as \tblOldSubmodular, \tblNewSubmodular, and
\tblZLP), and the  state-of-the-art evolutionary algorithm in \citet{Biesinger2016} (denoted as \tblBiesinger).
Note that \tblOldSubmodular can be seen as the state-of-the-art \BnC algorithm of \cite{Qi2024} based on submodular inequalities that is adapted to the \SCFLP under partially binary rule.
The four algorithms were tested on a benchmark testset of 36 instances of \cite{Biesinger2016} available from \url{https://www.ac.tuwien.ac.at/research/problem-instances}.
In these instances, the customer locations and the potential facility
locations were set to be identical, randomly chosen on a $[0, 100]\times[0, 100]$ Euclidean plane.
The attractiveness parameters were set to $v_{ij} = \frac{1}{d_{ij}+1}$,  where $d_{ij}$ is the distance  between customer $i\in [m]$ and facility $j\in [n]$. 
The number of locations $m$ and $n$ were set to $100$; the 
customer demands were chosen uniformly at random from $\{1,\ldots,10\}$; $p$ and $r$ were chosen from $p \in \{2,\ldots, 10\}$ and $r \in \{2,\ldots, 5\}$.

\afterpage{
\begin{landscape}
	\begin{table}[tbp]
		\centering
		\caption{Comparison of the \BnC algorithms based on formulations \eqref{prob:submod-aggr},  \eqref{prob:submod-disaggr-card1}, and \eqref{prob:ZLP}, and the evolutionary algorithm  of \citet{Biesinger2016}.}
		\label{tbl:Biesinger_3models}
%		\renewcommand{\arraystretch}{1.05}
%		\addtolength{\tabcolsep}{-0.7pt}	
		\resizebox{.88\linewidth}{!}{
			\begin{tabular}{lrrrrrrrrrrrrrrrrrrrrrrr}
				\toprule
				\multicolumn{1}{c}{\multirow{2}{*}{$m$}}
				& \multicolumn{1}{c}{\multirow{2}{*}{$n$}}
				& \multicolumn{1}{c}{\multirow{2}{*}{$p$}}
				& \multicolumn{1}{c}{\multirow{2}{*}{$r$}}
				& \multicolumn{6}{c}{\tblOldSubmodular}
				& \multicolumn{6}{c}{\tblNewSubmodular}
				& \multicolumn{6}{c}{\tblZLP}
				& \multicolumn{2}{c}{\tblBiesinger} 
				\\
				\cmidrule(r){5-10}
				\cmidrule(r){11-16}
				\cmidrule(r){17-22}
				\cmidrule(r){23-24} 
			 & & &	&    \tblO	& \tblTorG	&    \tblN	&   \tblrG	& \tblCutTime	& \tblCutNum	&    \tblO	& \tblTorG	&    \tblN	&   \tblrG	& \tblCutTime	& \tblCutNum	&    \tblO	& \tblTorG	&    \tblN	&   \tblrG	& \tblCutTime	& \tblCutNum	&     \tblO	& \tblT\\
				\midrule
100&100&2&2         	& 279.000&4.8	& 3510	& 8.1	&      2.2	&      285	& 279.000&3.2	& 0	& $<0.1$	&      1.3	&       24	&${ 279.000 }^*$&3.1	& 0	& $<0.1$	&      1.1	&        3& 278.931  & 529.0  \\
&&&3                	& 247.987&7.2	& 4817	& 10.2	&      4.0	&      692	& 247.987&3.4	& 0	& $<0.1$	&      1.5	&       37	&${ 247.987 }^*$&3.4	& 0	& $<0.1$	&      1.4	&        3& 247.946  & 606.0  \\
&&&4                	& 225.641&6.2	& 3650	& 10.2	&      3.3	&      651	& 225.641&3.4	& 0	& $<0.1$	&      1.5	&       42	&${ 225.641 }^*$&3.1	& 0	& $<0.1$	&      1.1	&        3& 225.640  & 560.0  \\
&&&5                	& 212.604&9.8	& 5609	& 11.0	&      5.9	&     1207	& 212.604&3.5	& 3	& $<0.1$	&      1.6	&       68	& 212.604 &4.1	& 3	& $<0.1$	&      1.4	&        4& 212.604  & 626.0  \\
&&3&2               	& 310.013&95.7	& 61207	& 8.1	&     50.9	&     3483	& 310.013&3.5	& 1	& $<0.1$	&      1.6	&       48	& 310.013 &3.1	& 0	& $<0.1$	&      1.2	&        3& 310.013  & 515.0  \\
&&&3                	& 279.000&32.2	& 44014	& 9.6	&     12.9	&      836	& 279.000&3.2	& 0	& $<0.1$	&      1.4	&       39	& 279.000 &3.0	& 0	& $<0.1$	&      1.1	&        3& 279.000  & 432.0  \\
&&&4                	& 255.704&51.6	& 43999	& 10.8	&     21.5	&     1681	& 255.704&3.2	& 0	& $<0.1$	&      1.4	&       45	&${ 255.704 }^*$&2.9	& 0	& $<0.1$	&      1.0	&        3& 255.072  & 410.0  \\
&&&5                	& 242.667&85.5	& 48959	& 11.0	&     31.6	&     3028	& 242.667&3.7	& 0	& $<0.1$	&      1.7	&       81	&${ 242.667 }^*$&3.0	& 0	& $<0.1$	&      1.0	&        3& 242.035  & 427.0  \\
&&4&2               	& 332.359&328.4	& 313955	& 7.5	&     96.1	&     2329	& 332.359&3.3	& 0	& $<0.1$	&      1.4	&       41	& 332.359 &2.9	& 0	& $<0.1$	&      1.1	&        3& 332.359  & 376.0  \\
&&&3                	& 302.296&323.5	& 283890	& 8.5	&     90.2	&     2444	& 302.296&3.3	& 0	& $<0.1$	&      1.5	&       69	&${ 302.296 }^*$&2.9	& 0	& $<0.1$	&      1.1	&        3& 302.217  & 354.0  \\
&&&4                	& 279.000&234.9	& 232033	& 9.1	&     65.8	&     2154	& 279.000&3.3	& 0	& $<0.1$	&      1.4	&       46	& 279.000 &2.8	& 0	& $<0.1$	&      1.1	&        3& 279.000  & 330.0  \\
&&&5                	& 265.917&442.8	& 259855	& 9.8	&    133.2	&     3845	& 265.917&3.3	& 0	& $<0.1$	&      1.5	&       52	& 265.917 &2.9	& 0	& $<0.1$	&      1.0	&        3& 265.917  & 365.0  \\
&&5&2               	& 344.869& (2.3)	& 955217	& 7.7	&    844.8	&    16784	& 345.259&4.2	& 7	& $<0.1$	&      2.2	&      262	&${ 345.259 }^*$&3.7	& 5	& $<0.1$	&      1.6	&        5& 345.116  & 444.0  \\
&&&3                	& 315.333& (1.6)	& 1522280	& 8.3	&    626.6	&    10332	& 315.333&3.4	& 0	& $<0.1$	&      1.5	&       88	&${ 315.333 }^*$&2.9	& 0	& $<0.1$	&      1.1	&        3& 313.582  & 362.0  \\
&&&4                	& 292.083& (1.5)	& 1482805	& 9.5	&    695.9	&    10131	& 292.083&3.3	& 3	& $<0.1$	&      1.4	&       82	&${ 292.083 }^*$&2.9	& 0	& $<0.1$	&      1.0	&        3& 291.000  & 330.0  \\
&&&5                	& 279.000& (2.3)	& 953946	& 9.8	&    716.4	&    15299	& 279.000&3.6	& 0	& $<0.1$	&      1.7	&      186	&${ 279.000 }^*$&2.9	& 0	& $<0.1$	&      1.1	&        3& 278.193  & 401.0  \\
&&6&2               	& 357.506& (3.4)	& 683772	& 7.6	&    983.8	&    22631	& 357.811&4.3	& 5	& $<0.1$	&      2.4	&      465	&${ 357.811 }^*$&3.5	& 5	& $<0.1$	&      1.5	&        4& 357.640  & 437.0  \\
&&&3                	& 326.703& (4.3)	& 592371	& 9.3	&    801.1	&    25807	& 326.974&4.1	& 9	& 0.1	&      2.2	&      286	&${ 326.974 }^*$&3.8	& 9	& 0.1	&      1.7	&        4& 325.250  & 386.0  \\
&&&4                	& 304.716& (4.0)	& 845635	& 9.6	&    644.3	&    18870	& 305.040&3.4	& 0	& $<0.1$	&      1.5	&       93	&${ 305.040 }^*$&2.8	& 0	& $<0.1$	&      1.1	&        3& 303.139  & 343.0  \\
&&&5                	& 291.059& (4.6)	& 743361	& 9.9	&    658.5	&    20328	& 291.958&3.5	& 0	& $<0.1$	&      1.6	&      155	&${ 291.958 }^*$&3.0	& 0	& $<0.1$	&      1.2	&        4& 290.754  & 400.0  \\
&&7&2               	& 368.437& (3.8)	& 392854	& 7.3	&   1299.5	&    35708	& 368.437&3.9	& 3	& $<0.1$	&      1.9	&      296	&${ 368.437 }^*$&3.7	& 3	& $<0.1$	&      1.6	&        7& 366.883  & 412.0  \\
&&&3                	& 339.007& (4.1)	& 729540	& 7.9	&    725.6	&    20130	& 339.052&3.7	& 0	& $<0.1$	&      1.8	&      219	&${ 339.052 }^*$&2.9	& 0	& $<0.1$	&      1.1	&        3& 335.827  & 348.0  \\
&&&4                	& 317.048& (4.1)	& 1042815	& 9.0	&    757.5	&    13501	& 317.157&3.4	& 0	& $<0.1$	&      1.5	&      127	&${ 317.157 }^*$&2.9	& 0	& $<0.1$	&      1.1	&        3& 315.167  & 298.0  \\
&&&5                	& 302.935& (4.9)	& 638860	& 9.3	&    867.1	&    21721	& 304.120&3.6	& 0	& $<0.1$	&      1.7	&      168	&${ 304.120 }^*$&2.8	& 0	& $<0.1$	&      1.1	&        3& 301.843  & 340.0  \\
&&8&2               	& 377.384& (3.9)	& 617941	& 6.4	&    939.6	&    22550	& 378.524&3.9	& 3	& $<0.1$	&      2.0	&      324	&${ 378.524 }^*$&3.8	& 3	& $<0.1$	&      1.7	&        8& 376.136  & 382.0  \\
&&&3                	& 348.788& (4.4)	& 698699	& 7.8	&    790.0	&    20000	& 349.765&3.9	& 0	& $<0.1$	&      1.8	&      205	&${ 349.765 }^*$&2.9	& 0	& $<0.1$	&      1.1	&        3& 347.421  & 344.0  \\
&&&4                	& 326.246& (4.8)	& 908199	& 8.6	&    533.1	&    16114	& 327.870&3.4	& 0	& $<0.1$	&      1.5	&      118	&${ 327.870 }^*$&2.9	& 0	& $<0.1$	&      1.1	&        3& 327.670  & 302.0  \\
&&&5                	& 314.509& (4.8)	& 683075	& 9.0	&    684.8	&    20090	& 314.819&3.6	& 0	& $<0.1$	&      1.6	&      210	&${ 314.819 }^*$&2.9	& 0	& $<0.1$	&      1.1	&        3& 314.168  & 340.0  \\
&&9&2               	& 384.826& (4.2)	& 482357	& 6.7	&    819.4	&    25571	& 386.550&5.7	& 19	& 0.1	&      3.6	&      750	&${ 386.550 }^*$&4.3	& 13	& $<0.1$	&      2.0	&        7& 385.354  & 316.0  \\
&&&3                	& 357.200& (4.7)	& 472440	& 7.8	&    788.6	&    27306	& 358.101&4.4	& 3	& $<0.1$	&      2.4	&      573	&${ 358.101 }^*$&2.9	& 0	& $<0.1$	&      1.1	&        3& 356.983  & 320.0  \\
&&&4                	& 335.637& (5.0)	& 855517	& 8.9	&    558.1	&    16119	& 336.412&3.7	& 0	& $<0.1$	&      1.6	&      175	&${ 336.412 }^*$&2.8	& 0	& $<0.1$	&      1.0	&        3& 335.919  & 318.0  \\
&&&5                	& 321.446& (5.9)	& 469833	& 9.1	&    718.2	&    29266	& 323.363&3.8	& 3	& $<0.1$	&      1.8	&      297	&${ 323.363 }^*$&3.3	& 3	& $<0.1$	&      1.3	&        4& 323.154  & 364.0  \\
&&10&2              	& 391.392& (4.6)	& 299722	& 7.1	&    936.1	&    36032	& 393.636&6.2	& 11	& 0.1	&      4.0	&     1353	&${ 393.636 }^*$&4.1	& 11	& 0.1	&      1.9	&        8& 388.068  & 303.0  \\
&&&3                	& 365.399& (4.7)	& 419297	& 7.8	&    962.6	&    24966	& 366.278&4.3	& 0	& $<0.1$	&      2.2	&      464	&${ 366.278 }^*$&2.9	& 0	& $<0.1$	&      1.1	&        3& 363.047  & 348.0  \\
&&&4                	& 344.142& (5.0)	& 772129	& 8.9	&    540.5	&    15908	& 344.217&3.7	& 0	& $<0.1$	&      1.7	&      252	&${ 344.217 }^*$&2.9	& 0	& $<0.1$	&      1.1	&        3& 343.982  & 304.0  \\
&&&5                	& 330.640& (5.3)	& 805437	& 8.9	&    524.6	&    16541	& 331.419&4.0	& 0	& $<0.1$	&      2.0	&      414	&${ 331.419 }^*$&2.8	& 0	& $<0.1$	&      1.1	&        3& 330.684  & 311.0  \\
\tblAve        & & &	&         	& 4845.1&{538155.6}	& 8.8	&    526.0	& {14565.0} &         	& 3.8&1.9	& $<0.1$	&      1.8	& 226.5	&         	& 3.1&1.5	& $<0.1$	&      1.2	& 3.7&&\\
\tblSolved     & & &	&       12	&         	&         	&         	&         	&         	&       36	&         	&         	&         	&         	&         	&       36	&         	&         	&         	&         	&       &&  \\

				\bottomrule
		\end{tabular}
	}
	\end{table}
	
\end{landscape}}

\cref{tbl:Biesinger_3models} summarizes the computational
results of settings \tblOldSubmodular, \tblNewSubmodular,
\tblZLP, and \tblBiesinger. 
For each setting of \tblOldSubmodular, \tblNewSubmodular, and 
\tblZLP, we report 
the objective value 
of the optimal solution or the best incumbent of the instance (\tblO),
the total CPU time in seconds (\tblT),
the number of explored nodes (\tblN), 
{the CPU time in seconds spent in separating the cuts (\tblCutTime),} and 
the number of added cuts (\tblCutNum).
For instances that cannot be solved to
optimality within the given time limit, we report under column \tblTorG the percentage optimality gap (\tblG) 
computed as $\frac{\tblUB - \tblLB}{\tblUB} \times 100\%$, where \tblUB and
 \tblLB denote the upper bound and the lower bound obtained at the end of the time limit. 
{To compare the strength of the three formulations \eqref{prob:submod-aggr}, \eqref{prob:submod-disaggr-card1},
	and \eqref{prob:ZLP}},
 we report the \LP relaxation gap at the root node (\tblrG), defined by $\frac{\tblO^* - \tblO_{\text{root}}}{\tblO^*}  \times 100\%$,
  where $\tblO^*$ and  $\tblO_{\text{root}}$ are {the optimal value of the \SCFLP} and the LP relaxation bound obtained at the root node, respectively.
 The smaller the \tblrG, the tighter the LP 
relaxation. 
At the end of the table, we report the summarized results.
Under setting \tblBiesinger, we report the objective value of the heuristic solution reported by \citet{Biesinger2016}.
For completeness, we also include the CPU time reported by \citet{Biesinger2016} (note that the experiments of \citet{Biesinger2016}  were carried out on an Intel Xeon Quadcore with 2.54 GHz using Cplex 12.5). 

We first observe from \cref{tbl:Biesinger_3models} that the performance of 
\tblOldSubmodular is fairly poor; overall, it can only solve $12$ instances among the $36$ instances to optimality within the time limit of $7200$ seconds.
{The reason behind this is that the submodular inequalities \eqref{ineq:submod-G} in formulation \eqref{prob:submod-aggr} are very weak} in terms of providing a poor \LP relaxation bound (see column \tblrG), which further forces \tblOldSubmodular to explore a huge search tree. 
{These results show that} a direct extension of the \BnC approach in \cite{Qi2024} based on submodular inequalities to our considered \SCFLP under partially binary rule indeed leads to a bad performance.
In sharp contrast, under setting \tblNewSubmodular, the \LP relaxation gaps are very small, typically smaller than $0.1\%$, showing that the proposed improved submodular inequalities \eqref{ineq:submod-explicit-G-D} are much more effective than the submodular inequalities \eqref{ineq:submod-G} in terms of providing a stronger LP relaxation bound.
This is attributed to
the favorable theoretical property of the improved submodular inequalities \eqref{ineq:submod-explicit-G-D} established in \cref{thm:submod-convexhull}; that is, compared with the submodular inequalities \eqref{ineq:submod-G} that can only characterize $\U_y$, 
the improved submodular inequalities \eqref{ineq:submod-explicit-G-D} are able to characterize $\conv(\U_y)$, an important substructure of the \SCFLP.
{Due to this advantage, \tblNewSubmodular outperforms \tblOldSubmodular by at least three orders of magnitude. In particular, equipped with the improved submodular inequalities,
\tblNewSubmodular solves all $36$ instances to optimality with an average CPU time being {$3.8$} seconds and an average number of nodes being $1.9$.  
}

Next, we compare the performance of settings \tblNewSubmodular and \tblZLP. 
We observe from \cref{tbl:Biesinger_3models} that the \LP relaxation gaps under settings \tblNewSubmodular and \tblZLP are {identical}, confirming that the \LP relaxations of problems \eqref{prob:ZLP} and \eqref{prob:submod-disaggr-card1} are indeed equivalent; see \cref{coro:ZLP=original_MINLP}.
Comparing the performance of the two proposed algorithms, we observe that \tblNewSubmodular requires a relatively large number of improved submodular cuts {\eqref{ineq:submod-explicit-G-D}} to ensure the convergence of the \BnC algorithm, whereas \tblZLP achieves convergence by adding only a small number of inequalities {\eqref{z-linear-1-no-y}}.
{This is reasonable as for a fixed $y \in \Y$, adding a single inequality \eqref{z-linear-1-no-y} in the underlying formulation \eqref{prob:ZLP} of \tblZLP has the same effect of adding the exponential number of   improved submodular inequalities \eqref{ineq:submod-explicit-G-D} in the underlying formulation \eqref{prob:submod-disaggr-card1} of \tblNewSubmodular;} see \cref{subsection:ZLP}.
Note that this comes at a cost of increasing the number of variables in the underlying formulation \eqref{prob:ZLP}. 
{Nevertheless,  on these small instances, \tblZLP converges slightly more  quickly than \tblNewSubmodular (as it requires significantly fewer cuts), {making it able} to compensate for the increase of the number of variables and achieve better overall performance.}
In the next subsection, we shall perform experiments on a testset of large-scale instances to further compare the performance of the proposed \tblNewSubmodular and \tblZLP.

Finally, we compare the performance of the proposed exact \BnC  algorithms with the state-of-the-art heuristic evolutionary algorithm of  \citet{Biesinger2016}. 
As shown in \cref{tbl:Biesinger_3models}, the evolutionary algorithm of \citet{Biesinger2016} can only identify an optimal solution for $7$ instances, although for the remaining ones, it can still identify a high-quality solution for the \SCFLP.
In contrast, the proposed \BnC algorithms (based on formulations \eqref{prob:submod-disaggr-card1} and \eqref{prob:ZLP}) can efficiently find the optimal solution for all instances. 
In \cref{tbl:Biesinger_3models}, we mark the $29$ instances, for which a better  solution is found {by the proposed exact algorithms},  by superscript ``$*$''.

\subsection{Comparison of 
	\eqref{prob:submod-disaggr-card1} with \eqref{prob:ZLP} on large scale instances
}

 \begin{figure}
 		\centering
 		\begin{minipage}{\textwidth}
 			\centering
 			\begin{subfigure}{0.240\textwidth}
 				\resizebox{\textwidth}{!}{\includegraphics{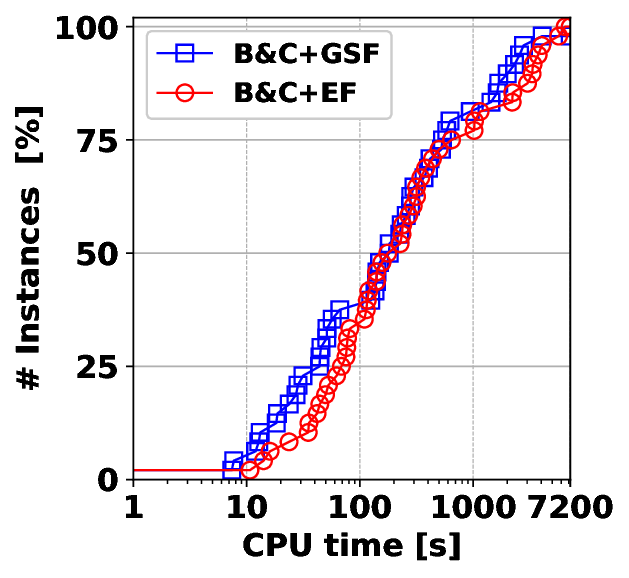}}
 				\caption{$m=500$}
 				\label{fig:time_m_500}
 			\end{subfigure}
 			 \hspace{0.01\textwidth} % 控制图片间距
			\begin{subfigure}{0.240\textwidth}
 				\resizebox{\textwidth}{!}{\includegraphics{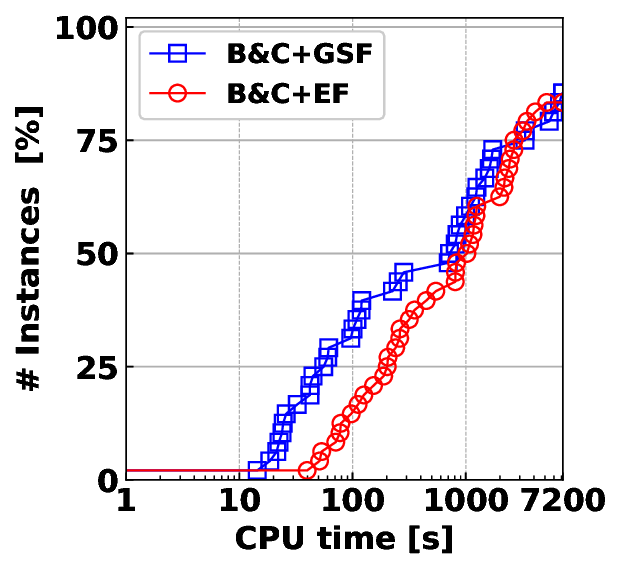}}
 				\caption{$m=800$}
 				\label{fig:time_m_800}
 			\end{subfigure}
 			 \hspace{0.01\textwidth}  
			\begin{subfigure}{0.240\textwidth}
 				\centering
 				\resizebox{\textwidth}{!}{\includegraphics{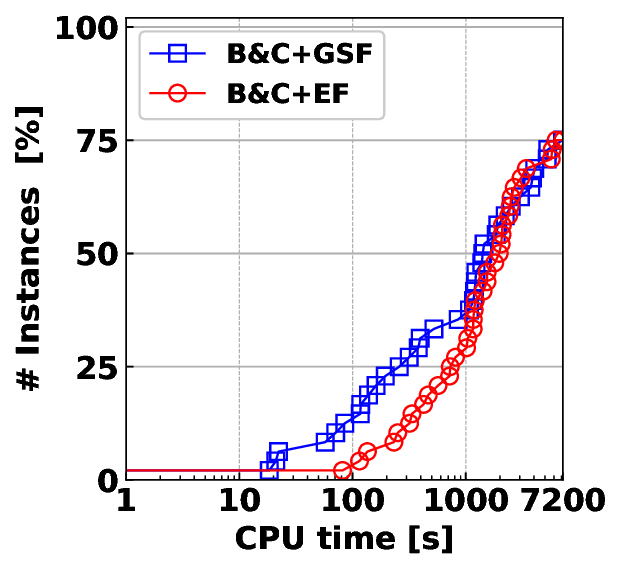}}
 				\caption{$m=1000$}
 				\label{fig:time_m_1000}
 			\end{subfigure}
 		\end{minipage}
 	\caption{Performance profiles of the CPU time on the large-scale \SCFLP instances, grouped by  $m$.}
 	\label{fig:m}
 \end{figure}

To further compare the performance of the proposed
 \BnC algorithms  based formulations \eqref{prob:submod-disaggr-card1} 
and \eqref{prob:ZLP} for solving large-scale \SCFLP{s}, we construct a testset of large-scale \SCFLP instances using a similar procedure as described in \cite{Qi2024}. 
Specifically, the locations of demand points and candidate
facility sites are randomly generated with integer coordinates on a $[0, 70]\times[0, 70]$  Euclidean plane.
The attractiveness parameters are set to $v_{ij} = e^{-0.1d_{ij}}$ based on the partially binary rule of \citet{Mendez-Vogel2023a},  
where $d_{ij}$ is the distance  between customer $i\in [m]$ and facility location $j\in [n]$. 
The number of customers $m$ and potential facility locations $n$ are all chosen from $\{500, 800, 1000\}$;
 the customer demands are chosen uniformly at random from $\{1,\ldots,10\}$;
  $p$ and $r$ are all taken from $\{2,5,10, 50\}$. 
In total, there are $144$ \SCFLP instances.

\cref{fig:m,fig:n,fig:p,fig:r}  plot the performance profiles of the CPU time under settings \tblNewSubmodular and \tblZLP 
grouped by $m$, $n$, $p$, and $r$, respectively. 
Detailed statistics of instance-wise computational results can be found in 
{Table \ref{tbl:New_ZLP_1000_merge}}
 of Appendix \ref{subsection:Appendix-table}. 

First, we compare the results of \tblNewSubmodular and \tblZLP for instances with different numbers of customers $m$ in \cref{fig:m}.
We observe that, as expected, increasing the number of customers $m$ makes the \SCFLP more difficult to be solved by both \tblNewSubmodular and \tblZLP.
In particular, when $m=500$, \tblNewSubmodular and \tblZLP can solve more than {$97\%$} and {$100\%$} instances to optimality within the time limit of $7200$ seconds,
{whereas when $m=1000$,} only approximately {$75\%$} instances can be solved to optimality with the same time limit. 
{Nevertheless, as shown in
	\cref{tbl:New_ZLP_1000_merge}
	in Appendix \ref{subsection:Appendix-table}, for the unsolved instances, the end gaps returned by \tblNewSubmodular and \tblZLP are usually smaller than {$1\%$}, demonstrating that  \tblNewSubmodular and \tblZLP can indeed be used to solve large-scale \SCFLP{s}.}
On the other hand, for the cases with $m=800, 1000$, we observe that due to the advantage of the much smaller number of variables in the underlying formulation \eqref{prob:submod-disaggr-card1}, \tblNewSubmodular generally performs better 
 than \tblZLP in terms of CPU time.
However, for instances with a small number of customers $m=500$, \tblZLP is able to converge more quickly and can solve slightly more instances to optimality than \tblNewSubmodular.

\begin{figure}
		\centering
		\begin{minipage}{\textwidth}
			\centering
			\begin{subfigure}{0.240\textwidth}
				\centering
				\resizebox{\textwidth}{!}{\includegraphics{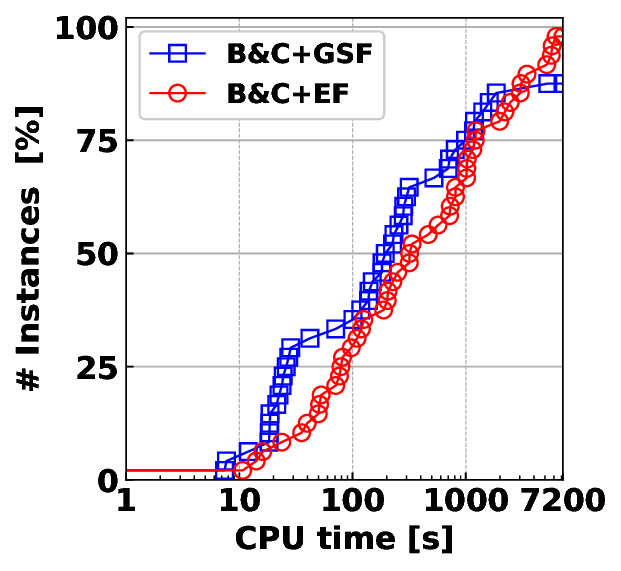}}
				\caption{ $n=500$}
				\label{fig:time_n_500}
			\end{subfigure}
			 \hspace{0.01\textwidth}
			\begin{subfigure}{0.240\textwidth}
				\centering
				\resizebox{\textwidth}{!}{\includegraphics{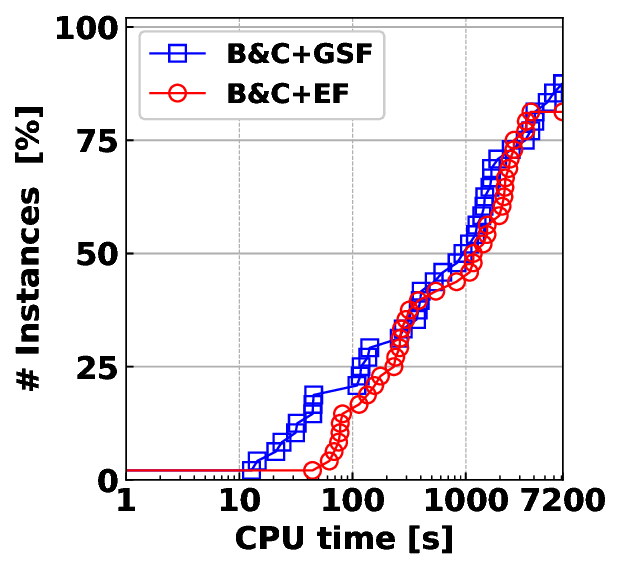}}
				\caption{ $n=800$}
				\label{fig:time_n_800}
			\end{subfigure}
			 \hspace{0.01\textwidth}
			\begin{subfigure}{0.240\textwidth}
				\centering
				\resizebox{\textwidth}{!}{\includegraphics{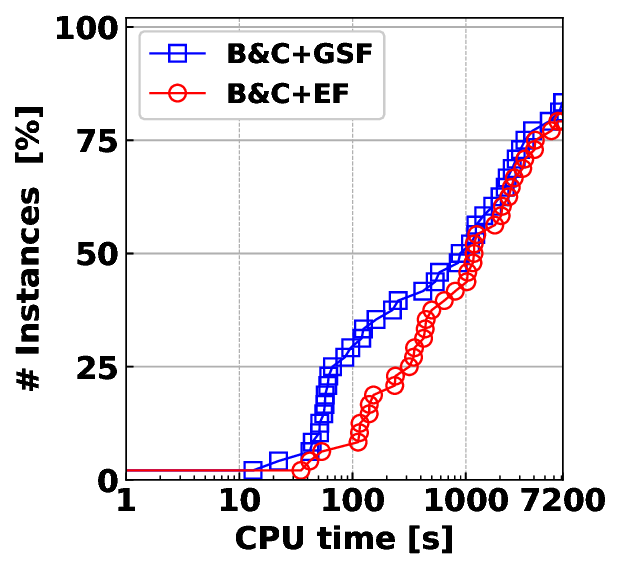}}
				\caption{ $n=1000$}
				\label{fig:time_n_1000}
			\end{subfigure}
			\hfill
				\end{minipage}

	\caption{Performance profiles of the CPU time on the large-scale \SCFLP instances, grouped by $n$.}
		  \label{fig:n}
\end{figure}

Next, we compare the performance of \tblNewSubmodular and \tblZLP for instances with different numbers of potential facility locations in \cref{fig:n}. 
We observe from  \cref{fig:n} that increasing the number of potential facility locations $n$ makes \SCFLP{s} more difficult to be solved by both  \tblNewSubmodular and \tblZLP.
On the other hand, in most cases,  \tblNewSubmodular performs better than \tblZLP, as it is based on the light-weight formulation \eqref{prob:submod-disaggr-card1}.
For instances with $n=500$, however,  \tblZLP can converge more quickly and thus can solve $10\%$ more instances to optimality.

\begin{figure}
		\centering
		\begin{minipage}{\textwidth}
			\centering
				\begin{subfigure}{0.240\textwidth}
				\centering
				\resizebox{\textwidth}{!}{\includegraphics{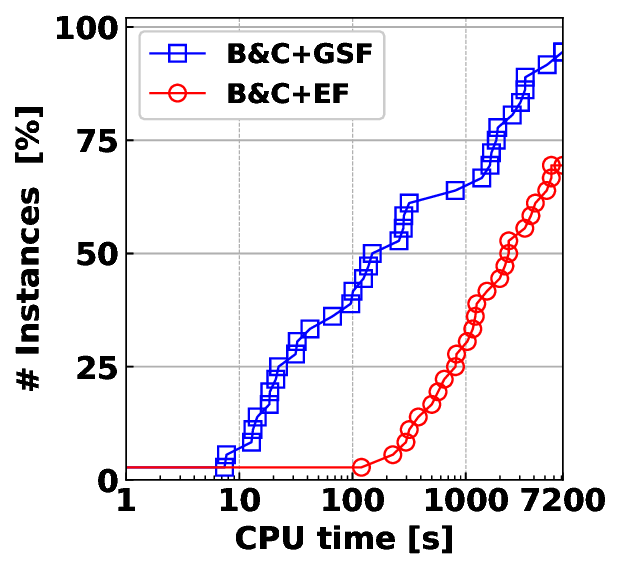}}
				\caption{ $p=2$}
				\label{fig:time_p_2}
			\end{subfigure}
			\hfill
			\begin{subfigure}{0.240\textwidth}
				\centering
				\resizebox{\textwidth}{!}{\includegraphics{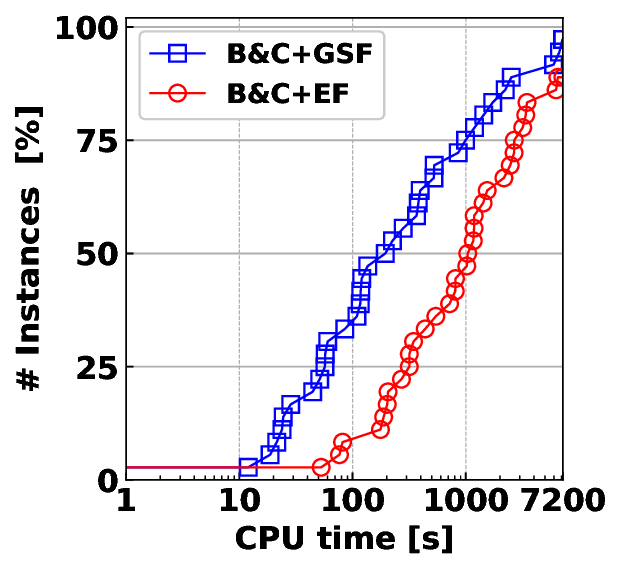}}
				\caption{ $p=5$}
				\label{fig:time_p_5}
			\end{subfigure}
			\hfill
			\begin{subfigure}{0.240\textwidth}
				\centering
				\resizebox{\textwidth}{!}{\includegraphics{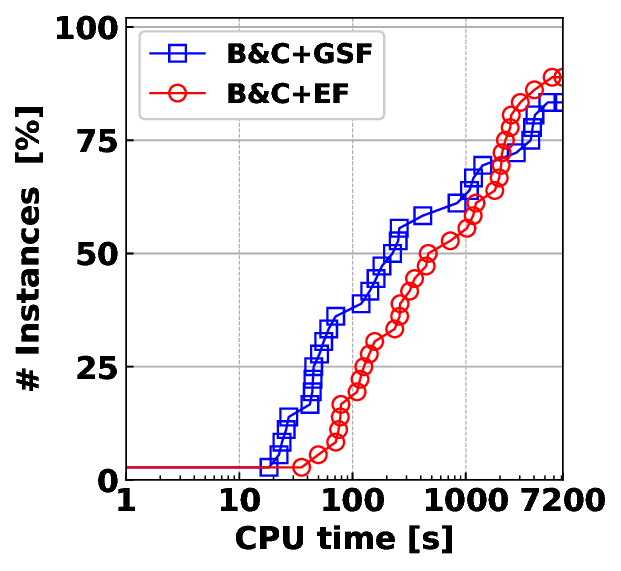}}
				\caption{ $p=10$}
				\label{fig:time_p_10}
			\end{subfigure}
			\hfill
			\begin{subfigure}{0.240\textwidth}
				\centering
				\resizebox{\textwidth}{!}{\includegraphics{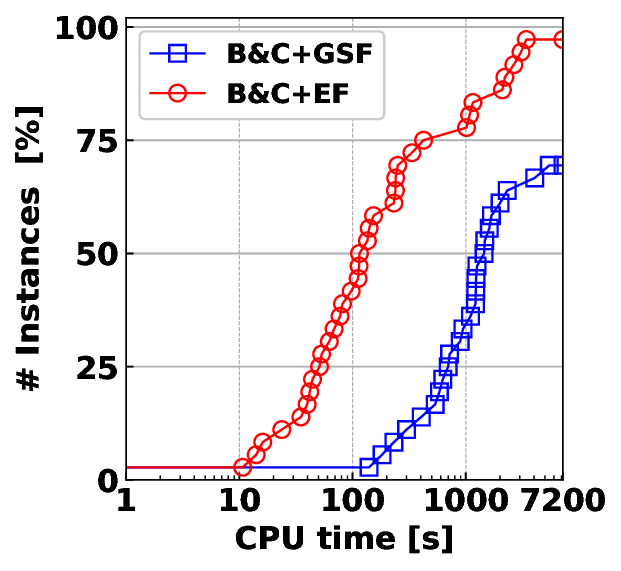}}
				\caption{ $p=50$}
				\label{fig:time_p_50}
			\end{subfigure}
		\end{minipage}
	\caption{Performance profiles of the CPU time on the large-scale \SCFLP instances, grouped by  $p$. }
		\label{fig:p}
\end{figure}

We now compare the performance of  \tblNewSubmodular and \tblZLP for instances with different numbers of the leader's open facilities  $p$. 
We observe from \cref{fig:p} that 
increasing the  number of leader's open facilities $p$ makes  \SCFLP{s} more difficult to be solved by \tblNewSubmodular.
In particular, when $p = 2$, more than $94\%$ instances can be solved to optimality by \tblNewSubmodular within the time limit of $7200$ seconds,
whereas  when $p = 50$, less than $70\%$ instances are solved to optimality by \tblNewSubmodular  within the same time limit.
In contrast, \tblZLP performs better as $p$ increases; \tblZLP can solve almost all instances with $p=50$ to optimality while it can only solve approximately $70\%$ instances with $p=2$ to optimality within the time limit of $7200$ seconds.
{Due to these results, we recommend using \tblNewSubmodular to solve \SCFLP{s} with small $p$ and using \tblZLP to solve \SCFLP{s} with large $p$.}

\begin{figure}
		\centering
		\begin{minipage}{\textwidth}
			\centering
			\begin{subfigure}{0.240\textwidth}
				\centering
				\resizebox{\textwidth}{!}{\includegraphics{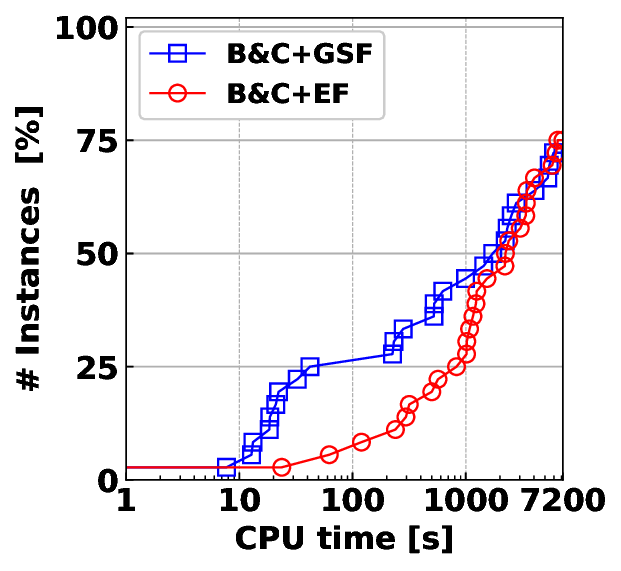}}
				\caption{ $r=2$}
				\label{fig:time_r_2}
			\end{subfigure}
			\hfill
			\begin{subfigure}{0.240\textwidth}
				\centering
				\resizebox{\textwidth}{!}{\includegraphics{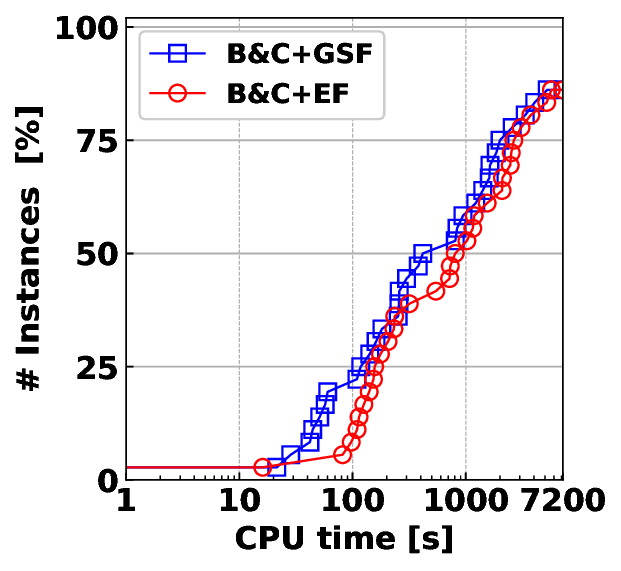}}
				\caption{ $r=5$}
				\label{fig:time_r_5}
			\end{subfigure}
			\hfill
			\begin{subfigure}{0.240\textwidth}
				\centering
				\resizebox{\textwidth}{!}{\includegraphics{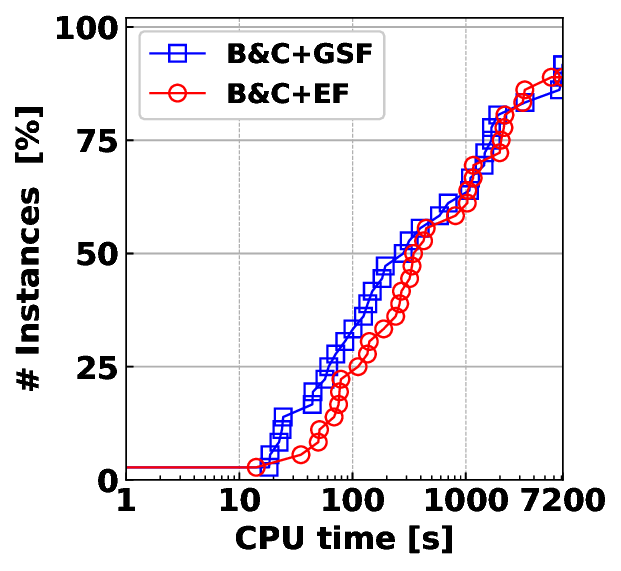}}
				\caption{ $r=10$}
				\label{fig:time_r_10}
			\end{subfigure}
			\hfill
			\begin{subfigure}{0.240\textwidth}
				\centering
				\resizebox{\textwidth}{!}{\includegraphics{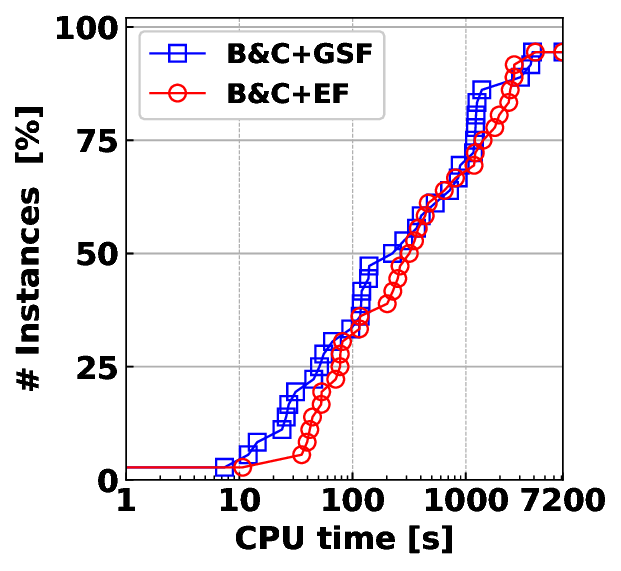}}
				\caption{ $r=50$}
				\label{fig:time_r_50}
			\end{subfigure}
		\end{minipage}
	\caption{Performance profiles of the CPU time on the large-scale \SCFLP instances, grouped by  $r$.}
	 \label{fig:r}
\end{figure}

Finally, we compare the performance of  \tblNewSubmodular and \tblZLP for instances with different numbers of the follower's open facilities  $r$. 
As shown in \cref{fig:r}, increasing the number of the follower's open facilities  $r$ makes \SCFLP{s} easier to be solved by both  \tblNewSubmodular and \tblZLP. 
On the other hand, we observe from \cref{fig:time_r_5,fig:time_r_10,fig:time_r_50} that the blue square lines corresponding to \tblNewSubmodular are higher than \tblZLP,
demonstrating that for \SCFLP instances with $r =5, 10, 50$, \tblNewSubmodular performs better than \tblZLP.
For  \SCFLP instances with $r=2$, although the blue square lines corresponding to \tblNewSubmodular are generally higher than \tblZLP,  \tblZLP, however, can solve slightly more instances to optimality than \tblNewSubmodular.

\section{Conclusions}\label{sec:Conclusions}

In this paper, we have considered the \SCFLP under partially binary rule and developed an efficient \BnC approach based on a single-level \MINLP formulation. 
In particular, we first established the submodularity of the multiple ratio function characterizing the leader's market share for each fixed follower's location choice, and extended the \BnC approach of \cite{Qi2024} based on submodular inequalities to our considered \SCFLP.
To address the challenge arising in the poor \LP relaxation of the underlying formulation, we then developed two new \MILP formulations with  much stronger \LP relaxations based on the polyhedral investigation of the mixed 0-1 set $\U_y$ induced by the hypograph of each multiple ratio function.
The first one is based on a newly proposed class of improved submodular inequalities, which include the classic submodular inequalities as special cases.
The second one is an extended formulation of the first one but only requires a single linear inequality to represent the aforementioned  multiple ratio function.
Two key features of the proposed \BnC algorithms, which make them able to solve large-scale \SCFLP{s}, are: (i) the \LP relaxations of the  \MILP formulations are very strong as the underlying linear inequalities are able to provide a linear characterization of $\conv(\U_y)$ (rather than $\U_y$, as achieved by classic submodular inequalities); and (ii) the separation of the exponential families of the linear inequalities can be conducted by solving a well-investigated $r$-median   problem (along with some acceleration techniques).   
By extensive computational experiments, we demonstrated that the proposed \BnC algorithms significantly outperform the state-of-the-art \BnC algorithm based on submodular inequalities of \cite{Qi2024} (adapted to our problem) and the evolutionary algorithm of  \citet{Biesinger2016}.
In particular, they are capable of finding optimal solutions for  \SCFLP instances with up to 1000 customers and facilities within a time limit of two hours.

\bibliography{shorttitles,SCFLP}
\newpage

\appendix
%{\noindent \Large \bf Online appendices for the paper
% ``An efficient branch-and-cut approach for the sequential competitive facility location problem under partially binary rule''}\\[20pt]
%\begin{center}	
%	{\large Yu-Qi Guo, Yan-Ru Wang,
%		Wei-Kun Chen, 
%		Yu-Hong Dai}
%\end{center}
%
%\vspace{0.5cm}
%
%\setcounter{page}{1}

\section{An example to show that the \LP relaxation  of problem \eqref{prob:submod-disaggr}  is 
	strictly stronger than that of problem \eqref{prob:submod-aggr}}
\label{appendix1}

Consider an example of the \SCFLP where
\[ v = 
\begin{bmatrix}
	1& 2& 1\\
	2&  1& 1\\
	1&  1& 2
\end{bmatrix}, \ 
w = \begin{bmatrix}
	1\\
	1
\end{bmatrix}, \ m = n=3, \ p = 2,\  r = 3.
\]
Then the \LP relaxation of problem \eqref{prob:submod-aggr} reduces to 
	\begin{equation*}
	\begin{aligned}
		\max_{\eta,\,  x \in [0,1]^3}  \{\eta\, : \,  \eqref{eq1} -\eqref{eq9} \}, ~\text{where}
	\end{aligned}
\end{equation*}

\vspace*{-1cm}
{\small
\begin{multicols}{2} 
	\begin{align}
		& x_1 + x_2 +x_3 = 2,\label{eq1}\\
		& \eta \leq \frac{3}{2}\label{eq2},\\
		& \eta \leq \frac{1}{6}x_2 + \frac{1}{6}x_3 + \frac{7}{6},\label{eq3} \\
		&\eta \leq \frac{1}{6}x_1 + \frac{1}{6}x_3 + \frac{7}{6},\label{eq4}\\
		&\eta \leq \frac{1}{6}x_1 + \frac{1}{6}x_2 + \frac{7}{6}\label{eq5},		
	\end{align}
	
	\columnbreak % 手动换列
	
	\begin{align}
		& \eta \leq \frac{1}{6}x_3 + \frac{4}{3},\label{eq6}\\ 
		&\eta \leq \frac{1}{6}x_2 + \frac{4}{3},\label{eq7}\\ 
		&\eta \leq \frac{1}{6}x_1 + \frac{4}{3}, \label{eq8}\\
		& \eta \leq \frac{7}{6}x_1 + \frac{7}{6}x_2 + \frac{7}{6}x_3. \label{eq9}
	\end{align}
\end{multicols}}%
\noindent An optimal solution of the above problem is $(\eta^1, x^1) = (\frac{25}{18}, \frac{2}{3}, \frac{2}{3}, \frac{2}{3})$. 
The \LP relaxation of \eqref{prob:submod-disaggr}, or equivalently of \eqref{prob:submod-disaggr-card1}, reduces to
\begin{equation*}
	\begin{aligned}
		\max_{\eta,\, x \in [0,1]^3} ~ \{\eta \, : \,  \eqref{eq1}-\eqref{eq21} \}, ~ \text{where}
	\end{aligned}
\end{equation*}
\vspace*{-1cm}

{\small
\begin{multicols}{2}
	\begin{align}
		&  \eta \leq \frac{1}{2} x_1 + \frac{1}{2} x_2 + \frac{1}{3} x_3 + \frac{5}{6},\label{eq10}\\
		& \eta \leq \frac{1}{2} x_1 + \frac{1}{3} x_2 + \frac{1}{2} x_3 + \frac{5}{6},\label{eq11}\\
		& \eta \leq \frac{1}{3} x_1 + \frac{1}{2} x_2 + \frac{1}{2} x_3 + \frac{5}{6},\label{eq12} \\ % 排序修正
		&  \eta \leq \frac{1}{2} x_1 + \frac{1}{3} x_2 + \frac{1}{3} x_3 + 1,\label{eq13}\\
		& \eta \leq  \frac{1}{3} x_1 +\frac{1}{2} x_2 + \frac{1}{3} x_3 + 1,\label{eq14}\\
		& \eta \leq \frac{1}{3} x_1 + \frac{1}{3} x_2 + \frac{1}{2} x_3 + 1\label{eq15}, % 排序修正
	\end{align}
	
	\columnbreak 
	
	\begin{align}
		&  \eta \leq \frac{5}{6} x_1 + \frac{5}{6} x_2 + \frac{2}{3} x_3 + \frac{1}{2},\label{eq16}\\
		& \eta \leq \frac{5}{6} x_1 + \frac{2}{3} x_2 + \frac{5}{6} x_3 + \frac{1}{2},\label{eq17}\\
		&\eta \leq \frac{2}{3} x_1 + \frac{5}{6} x_2+ \frac{5}{6} x_3 + \frac{1}{2}, \label{eq18}\\ % 排序修正
		&  \eta \leq \frac{1}{6} x_1 + \frac{1}{6} x_2 + \frac{1}{6} x_3 + 1, \label{eq19}\\
		&  \eta \leq \frac{1}{2} x_1 + \frac{1}{2} x_2 + \frac{1}{2} x_3 + \frac{2}{3}, \label{eq20}\\
		&  \eta \leq \frac{5}{6} x_1 + \frac{5}{6} x_2 + \frac{5}{6} x_3 + \frac{1}{3},\label{eq21}
	\end{align} 
\end{multicols}}%
\noindent  
An optimal solution of the above problem is  $(\eta^2, x^2) = (\frac{4}{3}, 1,1,0)$ (which satisfies the integrality constraints).
As $\eta^2 < \eta^1$, we conclude that 
the \LP relaxation of \eqref{prob:submod-disaggr} could be strictly stronger than that of \eqref{prob:submod-aggr}.

\section{Proof of \cref{thm:submod-convexhull}}\label{sec:appendix-pf-convexhull}

To establish the result in \cref{thm:submod-convexhull}, we need the following lemma.
\begin{lemma}\label{lemma:decompose-convex-hull}
	Given a sequence of convex functions $f_1, f_2, \ldots, f_m: [0,1]^n \rightarrow \R_{+}$, let
	\begin{equation}
		\CZ = \left\{(\eta,\zeta, x) \in \R \times \R^m \times
		\{0,1\}^n: \eta = \sum_{i=1}^m w_i \zeta_i,\  (\zeta_i, x) \in \CZ_i,\ \forall \  i \in [m] \right\},
	\end{equation}
	where $w_i\in  \mathbb{R}_+$ and $\CZ_i = \left\{(\zeta_i, x) \in \R \times
	\{0,1\}^n: \zeta_i \le  f_i(x) \right\}$ for $i \in [m]$.
	Then,
	\begin{equation}\label{convX}
		\conv(\CZ) = \left\{(\eta,\zeta, x) \in \R \times \R^m \times
		[0,1]^n: \eta = \sum_{i=1}^m w_i \zeta_i,\, (\zeta_i,x) \in \conv(\CZ_i),\ \forall \ i \in [m]\right\}.
	\end{equation}
\end{lemma}
\begin{proof}{Proof}
	Let $\CQ$ be the right-hand side of \eqref{convX}.
	%	We first show $\conv(\CZ) \subseteq Q$.
	Observe that
	\begin{align*}
		\CZ &= \left\{(\eta,\zeta, x) \in \R \times \R^m \times
		\{0,1\}^n: \eta = \sum_{i=1}^m w_i \zeta_i,\ (\zeta_i, x) \in \CZ_i,\ \forall \  i \in [m] \right\}\\
		&\!\!\!\overset{(a)}{=} \left\{(\eta,\zeta, x) \in \R \times \R^m \times
		[0,1]^n: \eta = \sum_{i=1}^m w_i \zeta_i,\ (\zeta_i, x) \in \CZ_i,\ \forall \  i \in [m] \right\}\\
		&\subseteq \left\{(\eta,\zeta, x) \in \R \times \R^m \times
		[0,1]^n: \eta = \sum_{i=1}^m w_i \zeta_i,\ (\zeta_i, x) \in \conv(\CZ_i),\ \forall \  i \in [m] \right\}= \CQ,
	\end{align*}
	where (a) follows $\CZ_i \subseteq \mathbb{R}\times \{0,1\}^n$ for all $i \in [m]$.
	This, together with the fact that $\CQ$ is a polyhedron, implies $\conv(\CZ)\subseteq \CQ$.
	
	Next, we prove $\CQ \subseteq \conv(\CZ)$.
	Letting  $(\eta, \zeta, x) \in \CQ$,
	%then it suffices to show $(\eta, \zeta, x) \in \conv(\CZ)$.
	then $(\zeta_1, x) \in \conv(\CZ_1)$.
	Therefore,
	there exist points $(\zeta_1^1, x^1), \dots, (\zeta_1^t, x^t) \in \CZ_1$ and scalars $\lambda_1, \dots, \lambda_t$ such that 
	\begin{equation}\label{tmp-convex-combination}
		(\zeta_1, x) = \sum_{\ell = 1}^t \lambda_{\ell} (\zeta_1^{\ell}, x^{\ell}), \quad 
		\sum_{\ell = 1}^t \lambda_{\ell} = 1, \quad
		\lambda_1, \dots, \lambda_t \ge 0.
	\end{equation}
	For each $\ell \in[t]$, we define $(\bar{\eta}^{\ell}, \bar{\zeta}^{\ell}, \bar{x}^{\ell}) $ as  
	$\bar{x}^{\ell} = x^{\ell}$, $\bar{\eta}^{\ell} = \sum_{i = 1}^m w_i \bar{\zeta}_i^{\ell}$, and 
	\begin{equation}\label{defzeta}
		\bar{\zeta}_i^{\ell} = \left\{
		\begin{aligned}
			& \zeta_1^{\ell}, && \text{if} ~ i = 1;\\
			& \zeta_i, && \text{if}~i \geq 2 ~\text{and}~ \sum_{\ell' = 1}^t \lambda_{\ell'} f_i(x^{\ell'}) = 0;\\
			&  \frac{f_i(x^{\ell}) \zeta_i}{\sum_{\ell' = 1}^t \lambda_{\ell'} f_i(x^{\ell'})}, &&\text{otherwise},
		\end{aligned}
		\right. \ \forall \ i \in [m].
	\end{equation}
	For $i \in\{2, \ldots, m\}$, we have 
	\begin{equation}\label{tmp-chain}
		\zeta_i \overset{(a)}{\le} f_i(x) \overset{(b)}{=} f_i\left(\sum_{\ell = 1}^t \lambda_{\ell} x^{\ell}\right)
		\overset{(c)}{\le} \sum_{\ell = 1}^t \lambda_{\ell} f_i(x^{\ell}), 
	\end{equation}
	where (a) follows from $(\zeta_i, x) \in \conv(\CZ_i) \subseteq \{ (\zeta_i, x)\, : \, \zeta_i \leq f_i(x) \}$ (as $(\eta, \zeta, x) \in \CQ$), 
%	\todo[inline]{YW: $\zeta$ should be revised into $\zeta_i$.}
	(b) follows from \eqref{tmp-convex-combination}, and (c) follows from the convexity of function $f_i$.
	If $\sum_{\ell' = 1}^t \lambda_{\ell'} f_i(x^{\ell'}) = 0$, then from 
	the definition of $\bar{\zeta}_i^{\ell}$ in \eqref{defzeta},
	the nonnegativity of $f_i$, and	\eqref{tmp-chain}, 
	we have $	\bar{\zeta}_i^{\ell} = \zeta_i \le
	\sum_{\ell' = 1}^t \lambda_{\ell'} f_i(x^{\ell'}) = 0 \le f_i(\bar{x}^{\ell})$;
	otherwise, from \eqref{tmp-chain}, we have ${ \frac{\zeta_i}{\sum_{\ell' = 1}^t \lambda_{\ell'} f_i(x^{\ell'})} \le 1}$ 
	and thus $\bar{\zeta}_i^{\ell} = \frac{\zeta_i}{\sum_{\ell' = 1}^t \lambda_{\ell'} f_i(x^{\ell'})} \cdot f_i(x^{\ell})
	\le f_i(x^\ell)= f_i(\bar{x}^{\ell})$.
		Thus, $\bar{\zeta}_i^{\ell} \le f_i(\bar{x}^{\ell})$ holds for all $i \in \{2, \ldots, m\}$, which, together with $\bar{\zeta}_i^{1} \le f_i(\bar{x}^{1})$, implies
	$(\bar{\eta}^{\ell}, \bar{\zeta}^{\ell}, \bar{x}^{\ell}) \in \CZ$. 
	Moreover, from the definition of $\bar{\zeta}^\ell_i$ in \eqref{defzeta}, it follows
	\begin{equation}\label{tmpeq0}
		\begin{aligned}
			& \zeta_i 
			{=} \sum_{\ell = 1}^t \lambda_{\ell} \zeta_i^{\ell} = \sum_{\ell = 1}^t \lambda_{\ell} \bar{\zeta}_i^{\ell} 
			\ \text{for}~  i = 1, 
			%			\ \ell = 1, \dots, t.
			\\
			& \zeta_i 
			= \sum_{\ell = 1}^t \lambda_{\ell} \zeta_i = \sum_{\ell = 1}^t \lambda_{\ell} \bar{\zeta}_i^{\ell}
			\ \text{for} \ i \geq 2 ~\text{with}~ \sum_{\ell' = 1}^t \lambda_{\ell'} f_i(x^{\ell'}) = 0, 
			%			\ \ell = 1, \dots, t.
			\\
			& \zeta_i = \frac{\sum_{\ell = 1}^t \lambda_{\ell} f_i(x^{\ell})}{\sum_{\ell' = 1}^t \lambda_{\ell'} f_i(x^{\ell'})} \zeta_i 
			= \sum_{\ell = 1}^t \lambda_{\ell} \cdot \frac{f_i(x^{\ell}) \zeta_i}{\sum_{\ell' = 1}^t \lambda_{\ell'} f_i(x^{\ell'})} 
			= \sum_{\ell = 1}^t \lambda_{\ell} \bar{\zeta}_i^{\ell} 
			\ \text{for} \ i \geq 2~\text{with}~ \sum_{\ell' = 1}^t \lambda_{\ell'} f_i(x^{\ell'}) > 0.
			%			\ell = 1, \dots, t.
		\end{aligned}
	\end{equation}
	Thus, 
	\begin{equation}\label{temp-sum}
		\eta = \sum_{i = 1}^m w_i \zeta_i 
		= \sum_{i = 1}^m w_i\sum_{\ell = 1}^t \lambda_{\ell} \bar{\zeta}_i^{\ell}
		= \sum_{\ell = 1}^t \lambda_{\ell}  \sum_{i = 1}^m  w_i\bar{\zeta}_i^{\ell}
		= \sum_{\ell = 1}^t \lambda_{\ell}  \bar{\eta}^{\ell}.
	\end{equation}
	Together with \eqref{tmp-convex-combination},  \eqref{tmpeq0}, and \eqref{temp-sum},  we have 
	\begin{equation*}
		(\eta, \zeta, x) = \sum_{\ell = 1}^t \lambda_{\ell} (\bar{\eta}^{\ell}, \bar{\zeta}^{\ell}, x^{\ell}).
	\end{equation*}
	This shows that $(\eta, \zeta, x)$ is a convex combination of finite points 
	$\{(\bar{\eta}^{\ell}, \bar{\zeta}^{\ell}, x^{\ell})\}_{\ell \in [t]}$ in $\CZ$,  and thus $(\eta, \zeta, x) \in \conv(\CZ)$.
	This completes the proof. 
\end{proof}

Now we are able to present the proof of \cref{thm:submod-convexhull}.
\begin{proof}{Proof of Theorem \ref{thm:submod-convexhull}}
	Given $y \in \Y$ and
	% let $\CU_y^{\SC\text{--L}}: =  \{(\eta, x) \in \R \times [0,1]^n: \eqref{ineq:submod-G-D}\}$, and $\CU_y^{\SC\text{--L}}: =  \{(\eta, x) \in \R \times [0,1]^n: \eqref{ineq:submod-G-D-card1}\}$.
	$i \in [m]$,  define
	%consider the following mixed integer set
	\begin{align}\label{def:Xiy}
		\CZ_{i} = \{(\zeta_i, x)\in \R \times \{0,1\}^n: \zeta_i \le  h_i(x,y)\}.
	\end{align}
	By \cref{thm:submodularity} (i), it follows $h_i(x,y)=  \max_{j \in [n]} c^y_{ij}x_j$, 
	which, together with \cref{thm:sub-Nemhauser1981-max}, implies $\conv(\CZ_{i}) = \CP_{i}$, 
	%		= \CZ_{i,y}^{\SC\text{--L}}$,
	where 
	\begin{equation}\label{def:Xiy-SC}
		\begin{aligned}
			\CP_{i}&:
			%					= \left\{ (\zeta_i, x)\in \R\times [0,1]^n: \zeta_i \le  c_{i\sigma_i\ell} +  \sum_{j=1}^{n} \left( c_{i\sigma_i(j)} -
			%				c_{i\sigma_i\ell}\right)^+ x_j, \ \forall \ \ell_i \in [n] \right\}.
			= \left\{ (\zeta_i, x)\in \R\times [0,1]^n: \zeta_i \le  c^y_{i\ell} +  \sum_{j=1}^{n} \left( c^y_{ij} -
			c^y_{i\ell}\right)^+ x_j, \ \forall \ \ell \in {[n+1]} \right\}. 
		\end{aligned}
	\end{equation}
	Let 
	\begin{align*}
		& \CZ := \left\{(\eta,\zeta, x) \in \R \times \R^m \times
		[0,1]^n: \eta = \sum_{i=1}^m w_i \zeta_i, ~
		(\zeta_i, x) \in \CZ_{i},\ \forall \  i \in [m] \right\},\\
		&\CP := \left\{(\eta, \zeta, x) \in \R \times \R^m \times
		[0,1]^n: \eta = \sum_{i=1}^m w_i \zeta_i, ~  (\zeta_i,x) \in \CP_{i},\ \forall \ i \in [m]\right\}.
	\end{align*}
	Then, $\U_y = \proj_{\eta,x} (\CZ)$ and
	$\CQ = \proj_{\eta,x} (\CP)$, where $\U_y$ is defined in \eqref{def:U_y_max} and $\CQ=\{(\eta, x) \in \R \times [0,1]^n: \eqref{ineq:submod-explicit-G-D}\} $ is the right-hand side of \eqref{convU}.
	As a result, 
	\begin{equation*}
		\begin{aligned}
			\conv(\CZ) &=\conv\left( \left\{(\eta,\zeta, x) \in \R \times \R^m \times
			\{0,1\}^n: \eta = \sum_{i=1}^m w_i \zeta_i \text{ and } 
			(\zeta_i, x) \in \CZ_{i},\ \forall \ i \in[m] \right\}\right)\\
			&\overset{(a)}{=} \left\{(\eta, \zeta, x) \in \R \times \R^m \times
			[0,1]^n: \eta = \sum_{i=1}^m w_i \zeta_i,\, (\zeta_i,x) \in \conv(\CZ_{i}),\ \forall \ i \in[m]\right\} \overset{(b)}{=} \CP, %\overset{(c)}{=} \CP_1 ,
			%		\left\{(\eta, \zeta, x) \in \R \times \R^m \times [0,1]^n: \eta = \sum_{i=1}^m w_i \zeta_i,\, (\zeta_i,x) \in \CZ_{i,y}^{\SC\text{--L}},\ \forall \ i \in[m]\right\},
		\end{aligned}
	\end{equation*}
	where (a) follows from \cref{lemma:decompose-convex-hull} and
	the fact that  $h_i(\cdot,y)$ is a form of $M(x)=\max_{j \in [n]} c_j x_j $
	with $c\in \R_+^n$, which is convex;
	and (b) 
	%		and (c) both
	follows from $\conv(\CZ_{i}) =\CP_i$.
	As a result,
	\begin{equation*}
		\conv(\U_y) = \conv(\proj_{\eta,x} (\CZ)) 
		\overset{(a)}{=} \proj_{\eta,x} (\conv(\CZ)) = \proj_{\eta,x} (\CP)\\
		= \CQ, % = \CU_y^{\SC\text{--L}},
	\end{equation*}
	where (a) follows from the fact that
	$\conv(\proj_x (\CZ')) = \proj_x (\conv(\CZ'))$ for any mixed integer set $\CZ'$ in the $x$ and $y$ space.  This completes the proof. 
\end{proof}

\section{Proof of \cref{thm:proj-Wy}}\label{sec:appendix-pf-proj-Wy}
\begin{proof}{Proof of \cref{thm:proj-Wy}}
	%\todo[inline]{YG: Note that $\CU_y$ is defined in \eqref{def:U_y_max_form}.}
	Let
	$(\hat{\eta},\hat{x}, \hat{z}^y) \in \W_{y}$. %and $\hat{x} \in  \CZ$.
	For $i \in [m]$, if $\hat{z}^y_{i\cdot} = \boldsymbol{0}^\top$, then $0 =  \sum_{j=1}^n c^y_{ij} \hat{z}^y_{ij} \leq \max_{j \in [n]} c^y_{ij} \hat{x}_j$; 
	otherwise, by \eqref{z-linear-4} and $\hat{z}^y \in \{0,1\}^{mn}$, $\hat{z}^y_{i\cdot} = \be_{{\ell_i}}^\top$ holds for some ${\ell_i}\in [n]$,
	and thus
	\begin{equation*} 
		\sum_{j=1}^n c^y_{ij} \hat{z}^y_{ij} =  c^y_{i{\ell_i}} \hat{z}^y_{i{\ell_i}} \overset{(a)}{\leq}  c^y_{i{\ell_i}} \hat{x}_{{\ell_i}} \leq \max_{j \in [n]} c^y_{ij} \hat{x}_j, 
	\end{equation*}
	where (a) follows from \eqref{z-linear-3}.
	In both cases, it follows $\sum_{j=1}^n c^y_{ij} \hat{z}^y_{ij} \leq \max_{j \in [n]} c^y_{ij} \hat{x}_j$ and $\hat{\eta} \leq \sum_{i=1}^m w_i \sum_{j=1}^n c^y_{ij} \hat{z}^y_{ij}\leq \sum_{i=1}^mw_i \max_{j \in [n]} c^y_{ij} \hat{x}_j $.
	Hence, $(\hat{\eta},\hat{x}) \in \CU_{y} $, and $\proj_{\eta, x} (\W_{y}) \subseteq \CU_{y}$.

	Now, suppose that $(\bar{\eta},\bar{x}) \in \CU_{y} $.
	If $\bar{x} = \boldsymbol{0}$, then $\bar{\eta} \leq 0$, and thus $(\bar{\eta}, \bar{x}, \boldsymbol{0}) \in \W_{y}$. 
	% and $\bar{x} \in \CZ$. 
	Otherwise, let $\bar{z}^y \in \{0,1\}^{mn}$ be such that $\bar{z}^y_{i\cdot}:=\be_{j_i}^\top$, where $j_i \in [n]$ satisfying $\bar{x}_{j_i}=1$ and $c^y_{ij_i} \bar{x}_{j_i} = \max_{j \in [n]} c^y_{ij} \bar{x}_j$. Hence, 
	%	$j_i \in \argmax c^y_{ij} \bar{x}_j$ for $i \in [m]$. 
	%	By definition and $c_{ij}^y \geq 0$ for all $i \in [m]$ and $j \in [n]$, it must follow $\bar{x}_{j_i}=1$ and 
	%\todo[inline]{YW: ``By definition \blue{of $\U$}''  (instead of ``by definition of $\bar{z}^y$'') will be clearer.}
	\begin{equation*}
		\max_{j \in [n]} c^y_{ij} \bar{x}_j  =  c^y_{ij_i} \bar{x}_{j_i} =   c^y_{ij_i} \bar{z}^y_{ij_i}  = \sum_{j=1}^n c^y_{ij} \bar{z}^y_{ij}.
	\end{equation*}
	Therefore,  $\bar{\eta}  \leq \sum_{i=1}^mw_i\max_{j \in [n]} c^y_{ij} \bar{x}_j =  \sum_{i=1}^m w_i \sum_{j=1}^n c^y_{ij} \bar{z}^y_{ij} $ and $(\bar{\eta}, \bar{x}, \bar{z}^y) \in \W_{y}$.
	This shows that $\CU_{y} \subseteq \proj_{\eta, x} (\W_{y}) $, which completes the proof. 
\end{proof}

\section{Proof of \cref{thm:exist-z^y}}\label{sec:appendix-pf-exist-z^y}
For each $i \in [m]$, let $\sigma_i(1),\ldots$, $\sigma_i(n)$ be a permutation of $[n]$ satisfying 
\begin{equation}\label{def:sigma-permutation-v}
	v_{i\sigma_i(1)} \geq \cdots \geq v_{i\sigma_i(n)}. 
\end{equation} 
To prove \cref{thm:exist-z^y}, we first provide a favorable property of $\{c_{ij}^y\}_{i\in [m],~j \in [n]}$ ($y \in \Y$); that is, the ordering of $\{c_{ij}^y\}_{i\in [m],~j \in [n]}$ is independent of $y \in \Y$.  
%\todo[inline]{YW: Please revise subscript ${(i,j)\in [m] \times [n]}$ into $i \in [m], \, j \in [n]$ \textbf{throughout the paper}.}
\begin{lemma} \label{lemma:sigma}
	For each $i \in [m]$, let $\sigma_i(\cdot)$ be the permutation satisfying  \eqref{def:sigma-permutation-v} and $\{c_{ij}^y\}_{i\in [m],~j \in [n]}$ be defined in \eqref{def:c^y_{ij}}. 
	Then for any $y \in \Y$, it follows 
	\begin{equation}
		c^{y}_{i\sigma_i(1)} \geq \cdots \geq c^{y}_{i\sigma_i(n)}.
	\end{equation}
\end{lemma}
\begin{proof}{Proof}
	The proof follows directly from $c_{ij}^y= \frac{v_{ij}} 
	{v_{ij} + \max_{k \in [n]}v_{ik}y_k}$ for $i \in [m]$ and $j \in [n]$ 
	and the fact that $v_{ij_1} \geq v_{ij_2}$ holds if and only if $\frac{v_{ij_1}} 
	{v_{ij_1} + \max_{k \in [n]}v_{ik}y_k} \geq \frac{v_{ij_2}} 
	{v_{ij_2} + \max_{k \in [n]}v_{ik}y_k} $ holds, where $j_1, j_2 \in [n]$.  
\end{proof}

Using \cref{lemma:sigma}, we can provide the optimal solutions for problems \eqref{def:h_i}, \eqref{def:h_i dual}, and \eqref{def:g_i}.
\begin{lemma}\label{rmk:optimal-solutions}
	Given $y \in  \Y$ and $i \in [m]$, 
	let $\sigma_i(1),\ldots$, $\sigma_i(n)$ be a permutation of $[n]$ satisfying \eqref{def:sigma-permutation-v}. 
	For $x \in [0,1]^n$, let
	\begin{equation}\label{def:k_i-star}
		k_i: = \left\{
		\begin{aligned}
			&1, \quad &&\text{if }~x_{\sigma_i(1)} = 1;\\ 
			&\max \left\{\ell \in [n] : \sum_{j=1}^\ell x_{\sigma_i(j)} < 1 \right\}, \quad && \text{otherwise}. 
		\end{aligned}
		\right.
	\end{equation}
	Then, we have the following three statements: 
	(i) an optimal solution $\hat{z}_{i\cdot}$ of problem \eqref{def:h_i} is given by 
	\begin{equation}\label{def:optimal_z}
		\hat{z}_{i\sigma_i(j)} = \left\{
		\begin{aligned}
			& x_{\sigma_i(j)}, \quad && \text{if}~j \leq k_i;\\
			& 1-\sum_{\ell=1}^{ k_i}x_{\sigma_i(\ell)}, \quad && \text{if}~j =  k_i+1;\\
			& 0, \quad && \text{otherwise},
		\end{aligned}
		\right. \ \forall \ j \in [n];
	\end{equation}
	with  $\sigma_i(n+1) := n+1$ and ${c_{i(n+1)}^y} :=0$, 	
%	\todo[inline]{YW: Revise $c_{n+1}^y$ into $c_{i(n+1)}^y$.}
	(ii) an optimal solution $(\hat{u}_{i\cdot}, \hat{w}_{i\cdot})$  of the dual problem \eqref{def:h_i dual} is given by
	\begin{equation}\label{def:optimal_u}
		\hat{u}_i = 	c^{y}_{i\sigma_i(k_i+1)}, \  \hat{w}_{ij} = (c^{y}_{ij}-c^{y}_{i\sigma_i(k_i+1)})^+, \ \forall \ j \in [n];
	\end{equation}  
	and (iii) an optimal solution of problem $g^y_i(x) = \min_{\ell_i \in [n+1]}g_{i,\ell_i}^{y}(x)$ 
	defined in \eqref{def:g_i} is given by  $\hat{\ell}_i = \sigma_i(k_i+1)$.
\end{lemma}
\begin{proof}{Proof}
	Statement (i) follows from \cref{lemma:sigma}.
	Statement (ii) follows from the fact that $(\hat{u}_{i\cdot}, \hat{w}_{i\cdot})$ is a feasible solution of problem \eqref{def:h_i dual} and the LP duality theory.
	Statement (iii) follows from \eqref{hdef} and statement (ii). 
\end{proof}

We are now ready to present the proof of \cref{thm:exist-z^y}.
\begin{proof}{Proof of Theorem \ref{thm:exist-z^y}}
	%We first consider
	Let $(\eta, x, \{\bar{z}^y\}_{y \in \Y})$ be an optimal solution of problem \eqref{prob:ZLP-pre}, $\hat{z}_{i\cdot}$ be defined as in \eqref{def:optimal_z} with 
	$\sigma_i(\cdot)$ being a permutation of $[n]$ satisfying \eqref{def:sigma-permutation-v}  for each $i \in [m]$, 
	and	 let $z^y=\hat{z}$ for $y \in \Y$. 
	%be defined as in \eqref{same-zy}.
	%	\begin{equation*}
		%		z^y = \hat{z},\ \forall \  y \in \Y.
		%	\end{equation*} 
	%	 with $\sigma$ replaced by $\bar{\sigma}$ for each $i \in [m]$. 
	By the definition of $\hat{z}$ in \eqref{def:optimal_z} and $x \in \{0,1\}^n$, $z^y =\hat{z} \in \{0,1\}^{mn}$ must hold for all $y \in \Y$. 
	In addition, constraints \eqref{z-linear-3} and \eqref{z-linear-4} hold at $(\eta, x, \{z^y\}_{y \in \Y})$.
	Hence, to prove that $(\eta, x, \{z^y\}_{y \in \Y})$ is an optimal solution of \eqref{prob:ZLP-pre}, it suffices to show \eqref{z-linear-1}.
	The latter is true since for any $y \in \Y$, it follows
	\begin{equation}\label{tempinequ}
		\eta \overset{(a)}{\leq}\sum_{i=1}^m w_i \sum_{j=1}^n c^y_{ij} \bar{z}^y_{ij} \overset{(b)}{\leq}  
		\sum_{i=1}^m w_i f^y_i(x)  \overset{(c)}{=} \sum_{i=1}^mw_i \sum_{j \in[n]}c_{ij}^y \hat{z}_{ij} \overset{(d)}{=} \sum_{i=1}^m w_i \sum_{j=1}^n c^y_{ij} z^y_{ij},
		%	 = \sum_{i=1}^m w_i \sum_{j=1}^n c^y_{ij} \hat{z}_{ij},
	\end{equation}%
	where (a) follows from the feasibility of $(\eta, x, \{\bar{z}^y\}_{y\in \Y})$,  
	(b) follows from the definition of $f_i^y(x)$ in \eqref{def:h_i} and and $\bar{z}^y_{i\cdot}$ is a feasible solution of problem \eqref{def:h_i},
	(c) follows from \cref{rmk:optimal-solutions} (i),
	and (d) follows from  $z^y=\hat{z}$ for $y \in \Y$. 
	This proves the statement for problem \eqref{prob:ZLP-pre}.
	The proof for the \LP relaxation of problem \eqref{prob:ZLP-pre} is analogous. 
\end{proof}

\section{Flow chart of the \BnC algorithms based on formulations \eqref{prob:submod-aggr},
	\eqref{prob:submod-disaggr-card1}, and \eqref{prob:ZLP}}\label{sec:appendix-BnCframework}
\begin{figure}[H]
	\centering
			\tikzstyle{startstop} = [rectangle, rounded corners, text centered, draw=black, fill=gray!50, minimum width=2cm, minimum height=0.8cm]
				\tikzstyle{process} = [rectangle, text centered, draw=black, fill=gray!20, minimum width=3cm, minimum height=0.8cm]
				\tikzstyle{decision} = [diamond, text centered, draw=black, fill=white, aspect=3.5]
				\tikzstyle{arrow} = [thick,->,>=stealth]
				\begin{tikzpicture} [node distance=55pt]
						\tikzstyle{every text node part} = [font=\footnotesize]
						\node (start) 	  [startstop, align=center] {start};
						\node (emptynodelist)	[decision, below of=start] {node list empty?};	
						\node (selectnode) [process, right= 65pt of start] {select a node};
						\node (lp) [process, below of=selectnode] {solve LP relaxation};			
						\node (vioBDcut) [decision, below of=lp, align=center] {inequalities violated?};
						\node (intLPsol)	[decision, below of=vioBDcut] {integral LP solution?};
						\node (branch) [process, below of=intLPsol] {branching};				
						\node (vault after empty 1) [right=8pt of emptynodelist] {};			
						\node (stop) [startstop, left=65pt of branch] {stop};									
						\node (addBDcut) [process, right=20pt of lp] {add violated inequalities};				
						\draw[arrow] (start) -- (emptynodelist);
						\draw[arrow] (emptynodelist) -- node[right]{Y} (stop);				
						\draw[arrow] (selectnode) -- (lp);
						\draw[arrow] (lp) -- (vioBDcut);
						\draw[arrow] (vioBDcut) -- node[right] {N}(intLPsol);
						\draw[arrow] (intLPsol) -- node[right] {N} (branch);
						\draw[arrow] (addBDcut) -- (lp);				
						\draw[arrow] (vioBDcut) -- node[above] {Y} (vioBDcut -| addBDcut) --  (addBDcut);				
						\draw[arrow] (emptynodelist.east) -| node[right] {N} ($(emptynodelist)!.5!(selectnode)$) |- (selectnode) ;				
						\coordinate (proj1) at ($(addBDcut.east) + (0.1, 0)$);
						\coordinate (proj2) at ($(addBDcut.east) + (0.1, -1)$);
						\draw[arrow] (intLPsol) -- node[above] {Y} ($(proj1)!(intLPsol)!(proj2)$);
						\draw[arrow] (branch) -| ($(addBDcut.east)+ (0.1, 0)$)  |- ($(start.west) + (-0.5,0.5)$) -| ($(emptynodelist.west) + (-0.5, 0)$) |- (emptynodelist.west);
				\end{tikzpicture}
	\caption{Flow chart of the \BnC algorithm based on formulations \eqref{prob:submod-aggr},
		\eqref{prob:submod-disaggr-card1}, and \eqref{prob:ZLP}.} \label{figure:framework}
\end{figure}
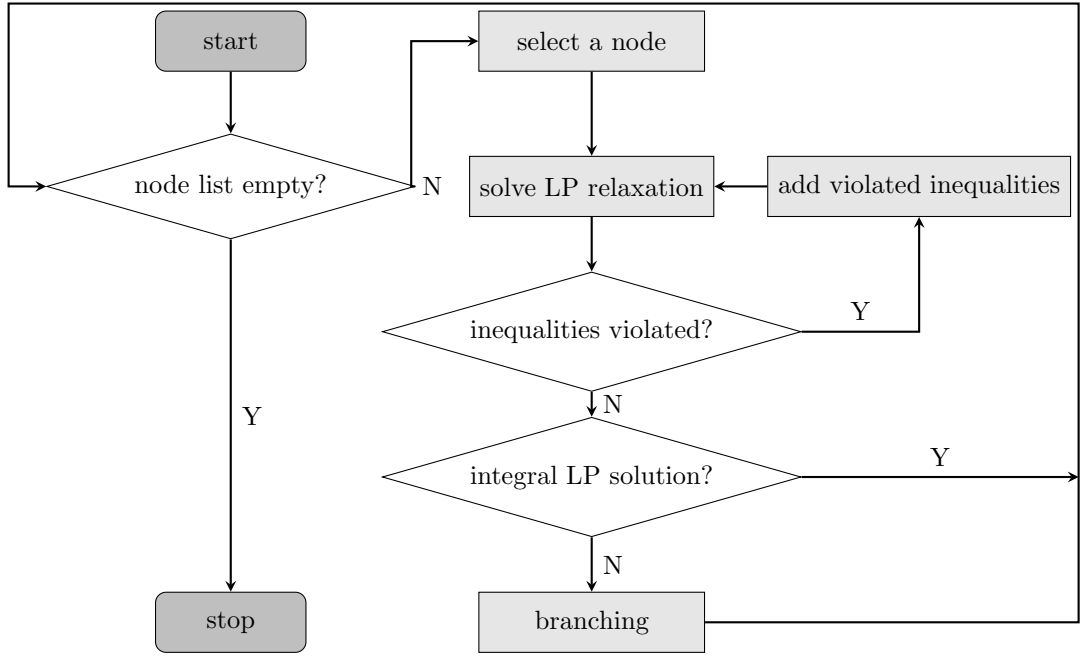

\section{Proof of \cref{thm:newsub-sepa}}\label{sec:appendix-newsub-sepa}
 \begin{proof}{Proof of \cref{thm:newsub-sepa}}
 	Note that for a fixed $y$, the inner subproblem of \eqref{sepa:newsub} is
 	{\small \begin{equation}\label{sepa:layer-2-app}
 			\begin{aligned}
 				\min_{(\ell_1, \dots, \ell_m) \in [n+1]^m} \sum_{i=1}^m w_i \left( c^y_{i\ell_i} +  \sum_{j=1}^{n} \left( c^y_{ij} - 
 				c^y_{i\ell_i}\right)^+ x^*_j\right)
 			 &=\sum_{i=1}^mw_i \min_{\ell_i \in [n+1]} \left( c^y_{i\ell_i} +  \sum_{j=1}^{n} \left( c^y_{ij} - 
 			 c^y_{i\ell_i}\right)^+ x^*_j\right) \\
 				&\overset{(a)}{=}  \sum_{i=1}^m w_i \left( c^y_{i\sigma_i(k_i+1)} +  \sum_{j=1}^{n} \left( c^y_{ij} - 
 				c^y_{i\sigma_i(k_i+1)}\right)^+ x^*_j\right), 
 			\end{aligned}
 	\end{equation}}%
 	where (a) follows from \cref{rmk:optimal-solutions} (iii) {in Appendix \ref{sec:appendix-pf-exist-z^y}} and the definition of $g_{i,\ell_i}^{y}(x)$ in \eqref{def:g_i} for $i \in [m]$ and $\ell_i \in [n]$.
 	This shows the first equality in \eqref{pmp-newsub-layer0}.
 	As for the second equality in \eqref{pmp-newsub-layer0}, 
 	it follows from the definition of 
 	$\{c_{ij}^y\}_{i\in [m],~j \in [n]}$ in \eqref{def:c^y_{ij}} that 
 	{\small
 		 \begin{equation*}%\label{equ:f_k_i}
 			\begin{aligned}
 				&c^y_{i\sigma_i(k_i+1)} +  \sum_{j=1}^{n} \left( c^y_{ij} - 
 				c^y_{i\sigma_i(k_i+1)}\right)^+ x^*_j \\
 				&\quad=\frac{v_{i\sigma_i(k_i+1)}}{v_{i\sigma_i(k_i+1)} + \max_{k \in [n]}v_{ik}y_k} +  \sum_{j=1}^{k_i} \left( \frac{v_{i\sigma_i(j)}}{v_{i\sigma_i(j)} +
 					\max_{k \in [n]}v_{ik}y_k} - \frac{v_{i\sigma_i(k_i+1)}}{v_{i\sigma_i(k_i+1)} + \max_{k \in [n]}v_{ik}y_k}\right) x^*_{\sigma_i(j)}\\
 				&\quad=\frac{v_{i\sigma_i(k_i+1)}}{v_{i\sigma_i(k_i+1)}  + \max_{k \in [n]}v_{ik}y_k}\left(1-\sum_{j=1}^{k_i}x^*_{\sigma_i(j)}\right) +
 				\sum_{j=1}^{k_i}  \frac{v_{i\sigma_i(j)}}{v_{i\sigma_i(j)} +\max_{k \in [n]}v_{ik}y_k} x^*_{\sigma_i(j)}\\
 				&\quad\overset{(a)}{=}\left(\min_{k \in [n]}\frac{v_{i\sigma_i(k_i+1)}}{v_{i\sigma_i(k_i+1)} + v_{ik}}y_k \right) \left(1-\sum_{j=1}^{k_i}x^*_{\sigma_i(j)}\right)  +
 				\sum_{j=1}^{k_i} \min_{k \in [n]}\frac{ x^*_{\sigma_i(j)}v_{i\sigma_i(j)}}{v_{i\sigma_i(j)} + v_{ik}}y_k \\
 				&\quad{=}\min_{k \in [n]}\frac{\left(1-\sum_{j=1}^{k_i}x^*_{\sigma_i(j)}\right)v_{i\sigma_i(k_i+1)}}{v_{i\sigma_i(k_i+1)} + v_{ik}}y_k +
 				\min_{k \in [n]} \sum_{j=1}^{k_i} \frac{ x^*_{\sigma_i(j)}v_{i\sigma_i(j)}}{v_{i\sigma_i(j)} + v_{ik}}y_k \\
 				&\quad\overset{(b)}{=}\min_{k \in [n]} \left(\frac{\left(1-\sum_{j=1}^{k_i}x^*_{\sigma_i(j)}\right)v_{i\sigma_i(k_i+1)}}{v_{i\sigma_i(k_i+1)} + v_{ik}} +
 				\sum_{j=1}^{k_i} \frac{ x^*_{\sigma_i(j)}v_{i\sigma_i(j)}}{v_{i\sigma_i(j)} + v_{ik}}  \right)y_k  = \min_{k \in [n]} b_{ik} y_k,
 			\end{aligned}
 	\end{equation*}}%		
 	where  $(a)$ and $(b)$ follow from 
 	$\argmin_{k \in [n]} \frac{\beta}{\alpha+v_{ik}}y_k = \argmax_{k \in [n]}v_{ik} y_k$  for all $(\alpha,\beta) 
 	\in  \R^2_{+}$ (as $1-\sum_{j=1}^{k_i}x^*_j \geq 0$). 
 	Therefore,
 	\begin{equation*}
 		\sum_{i=1}^m w_i \left( c^y_{i\sigma_i(k_i+1)} +  \sum_{j=1}^{n} \left( c^y_{ij} - 
 		c^y_{i\sigma_i(k_i+1)}\right)^+ x^*_j\right) =\sum_{i=1}^m w_i \min_{k \in [n]} b_{ik} y_k.	\qedhere
 	\end{equation*}
 	
 \end{proof}

\section{Detailed statistics of instance-wise computational results of
	 the \BnC algorithms based on \eqref{prob:submod-disaggr-card1} and  \eqref{prob:ZLP} on large-scale \SCFLP instances}\label{subsection:Appendix-table}
	 
The detailed statistics of instance-wise computational results is summarized in \cref{tbl:New_ZLP_1000_merge} in the next page.
\newpage
	\begin{longtable}{lrrrrrrrrrrrrrrr}
		\caption{Detailed statistics of instance-wise computational results of
			the \BnC algorithms based on \eqref{prob:submod-disaggr-card1} and  \eqref{prob:ZLP} on large-scale \SCFLP instances.} 
			\label{tbl:New_ZLP_1000_merge}\\
		\toprule
		\multicolumn{1}{c}{\multirow{2}{*}{$m$}} 
		& \multicolumn{1}{c}{\multirow{2}{*}{$n$}} 
		& \multicolumn{1}{c}{\multirow{2}{*}{$p$}} 
		& \multicolumn{1}{c}{\multirow{2}{*}{$r$}} 
		& \multicolumn{6}{c}{\tblNewSubmodular} 
		& \multicolumn{6}{c}{\tblZLP} 
		\\
		\cmidrule(r){5-10}
		\cmidrule(r){11-16}
		& & &	&    \tblO	& \tblTorG	&    \tblN	&   \tblrG	& \tblCutTime	& \tblCutNum	&    \tblO	& \tblTorG	&    \tblN	&   \tblrG	& \tblCutTime	& \tblCutNum\\
		\midrule
		\endfirsthead
		
		\toprule
		\multicolumn{1}{c}{\multirow{2}{*}{$m$}} 
		& \multicolumn{1}{c}{\multirow{2}{*}{$n$}} 
		& \multicolumn{1}{c}{\multirow{2}{*}{$p$}} 
		& \multicolumn{1}{c}{\multirow{2}{*}{$r$}} 
		& \multicolumn{6}{c}{\tblNewSubmodular} 
		& \multicolumn{6}{c}{\tblZLP} \\
		\cmidrule(r){5-10}
		\cmidrule(r){11-16}
		& & &	&    \tblO	& \tblTorG	&    \tblN	&   \tblrG	& \tblCutTime	& \tblCutNum	&    \tblO	& \tblTorG	&    \tblN	&   \tblrG	& \tblCutTime	& \tblCutNum\\
		\midrule
		\endhead
		
		\midrule
		\multicolumn{14}{r}{\small\emph{continued on next page}} \\
		\midrule
		\endfoot
		
		\bottomrule
		\endlastfoot
		
		%======= 表格数据开始 =======
500&500&2&2         	& 1358.000&7.7	& 3	& $<0.1$	&      5.7	&       40	& 1358.000&119.8	& 0	& $<0.1$	&      5.8	&        3\\
&&&5 	               	& 882.851&256.6	& 28	& 0.5	&    251.9	&     1911	& 882.851&1153.0	& 27	& 0.5	&     56.2	&       25\\
&&&10               	& 707.142&149.4	& 22	& 0.2	&    146.2	&      994	& 707.142&1030.7	& 14	& 0.2	&     52.8	&       25\\
&&&50               	& 500.237&7.4	& 3	& $<0.1$	&      5.3	&       74	& 500.237&226.7	& 0	& $<0.1$	&      4.4	&        4\\
&&5&2               	& 1827.562&279.9	& 51	& $<0.1$	&    242.5	&     5878	& 1827.562&1019.9	& 58	& $<0.1$	&     87.7	&       13\\
&&&5                	& 1358.000&28.4	& 7	& $<0.1$	&     25.4	&      488	& 1358.000&81.6	& 3	& $<0.1$	&     10.4	&        3\\
&&&10               	& 1121.053&18.7	& 3	& $<0.1$	&     15.9	&      587	& 1121.053&76.6	& 7	& $<0.1$	&     12.2	&        6\\
&&&50               	& 826.549&12.0	& 3	& $<0.1$	&      9.4	&      415	& 826.549&52.8	& 0	& $<0.1$	&      3.9	&        4\\
&&10&2              	& 2006.534& (0.1)	& 71	& $<0.1$	&    871.4	&    60323	& 2007.298&3027.3	& 180	& $<0.1$	&    370.7	&       16\\
&&&5                	& 1594.947&181.7	& 3	& $<0.1$	&     73.1	&     6293	& 1594.947&109.6	& 3	& $<0.1$	&      9.4	&        6\\
&&&10               	& 1358.000&18.3	& 3	& $<0.1$	&     13.4	&      834	& 1358.000&49.8	& 1	& $<0.1$	&      7.3	&        6\\
&&&50               	& 1036.268&27.3	& 3	& $<0.1$	&     21.8	&     1936	& 1036.268&35.5	& 0	& $<0.1$	&      3.4	&        3\\
&&50&2              	& 2215.763&231.7	& 0	& $<0.1$	&    116.2	&    24271	& 2215.763&23.7	& 0	& $<0.1$	&      6.5	&        4\\
&&&5                	& 1889.451&300.0	& 3	& $<0.1$	&     93.6	&    19075	& 1889.451&16.1	& 0	& $<0.1$	&      6.8	&        4\\
&&&10               	& 1679.732&182.1	& 0	& $<0.1$	&     43.9	&     8628	& 1679.732&14.1	& 0	& $<0.1$	&      5.7	&        4\\
&&&50               	& 1358.000&140.0	& 0	& $<0.1$	&     35.6	&     6820	& 1358.000&10.7	& 0	& $<0.1$	&      3.3	&        3\\
&800&2&2            	& 1358.000&12.8	& 3	& $<0.1$	&     10.7	&       41	& 1358.000&295.9	& 0	& $<0.1$	&      8.7	&        3\\
&&&5                	& 891.549&1632.7	& 94	& 0.8	&   1546.0	&     7861	& 891.549&3756.3	& 75	& 1.3	&    218.2	&       29\\
&&&10               	& 701.716&280.8	& 13	& 0.2	&    274.9	&     1565	& 701.716&2213.7	& 15	& 0.1	&     55.4	&       23\\
&&&50               	& 502.442&31.4	& 5	& $<0.1$	&     29.2	&      243	& 502.442&380.6	& 3	& $<0.1$	&      8.5	&        8\\
&&5&2               	& 1818.286&525.2	& 45	& $<0.1$	&    445.4	&     6564	& 1818.286&3383.1	& 107	& $<0.1$	&    283.8	&       15\\
&&&5                	& 1358.000&44.4	& 3	& $<0.1$	&     39.9	&     1008	& 1358.000&176.3	& 0	& $<0.1$	&      8.8	&        3\\
&&&10               	& 1109.254&136.1	& 13	& $<0.1$	&    122.5	&     2111	& 1109.254&270.2	& 9	& $<0.1$	&     19.7	&        5\\
&&&50               	& 816.707&367.6	& 16	& $<0.1$	&    310.2	&     7137	& 816.707&317.1	& 7	& $<0.1$	&     15.1	&       12\\
&&10&2              	& 2012.616&4079.7	& 145	& 0.2	&   1037.6	&    29625	& 2012.616&2236.9	& 116	& $<0.1$	&    318.8	&       16\\
&&&5                	& 1606.703&141.7	& 13	& 0.1	&    110.6	&     4964	& 1606.703&156.7	& 3	& $<0.1$	&     20.2	&        8\\
&&&10               	& 1358.000&23.8	& 3	& $<0.1$	&     18.3	&      745	& 1358.000&75.3	& 0	& $<0.1$	&      6.3	&        3\\
&&&50               	& 1041.993&45.3	& 3	& $<0.1$	&     38.5	&     2484	& 1041.993&77.9	& 0	& $<0.1$	&      5.6	&        4\\
&&50&2              	& 2213.550&626.7	& 6	& $<0.1$	&    274.5	&    45127	& 2213.550&62.3	& 5	& $<0.1$	&     22.6	&        5\\
&&&5                	& 1899.261&944.8	& 11	& $<0.1$	&    153.2	&    19684	& 1899.261&114.1	& 19	& $<0.1$	&     62.8	&       11\\
&&&10               	& 1674.007&1684.0	& 58	& $<0.1$	&    219.1	&    20451	& 1674.007&68.7	& 28	& $<0.1$	&     33.7	&        3\\
&&&50               	& 1358.000&403.9	& 7	& $<0.1$	&     87.4	&    12133	& 1358.000&44.3	& 13	& $<0.1$	&     16.0	&        9\\
&1000&2&2           	& 1358.000&13.2	& 3	& $<0.1$	&     11.2	&       45	& 1358.000&500.6	& 0	& $<0.1$	&     10.6	&        3\\
&&&5                	& 891.948&2572.9	& 94	& 0.5	&   2363.9	&    11398	& 891.948&5697.3	& 82	& 0.5	&    310.3	&       31\\
&&&10               	& 701.716&124.8	& 9	& 0.1	&    121.3	&      695	& 701.716&3320.0	& 13	& 0.1	&     43.1	&       16\\
&&&50               	& 501.455&66.5	& 9	& $<0.1$	&     64.3	&      233	& 501.455&645.4	& 5	& $<0.1$	&     15.9	&       13\\
&&5&2               	& 1818.882&1437.0	& 99	& $<0.1$	&   1187.0	&    11506	& 1818.882&6501.3	& 151	& $<0.1$	&    467.8	&       14\\
&&&5                	& 1358.000&51.2	& 5	& $<0.1$	&     47.6	&      432	& 1358.000&317.0	& 3	& $<0.1$	&     21.2	&        8\\
&&&10               	& 1109.686&57.0	& 5	& $<0.1$	&     52.7	&      813	& 1109.686&345.6	& 3	& $<0.1$	&     14.2	&        7\\
&&&50               	& 815.466&224.9	& 9	& $<0.1$	&    185.6	&     3016	& 815.466&437.9	& 9	& $<0.1$	&     20.1	&       13\\
&&10&2              	& 2013.183&2787.7	& 56	& 0.1	&    815.2	&    24979	& 2013.183&4039.8	& 124	& $<0.1$	&    401.2	&       15\\
&&&5                	& 1606.279&417.2	& 5	& $<0.1$	&    200.4	&    10134	& 1606.279&235.7	& 5	& $<0.1$	&     24.5	&        7\\
&&&10               	& 1358.000&44.1	& 3	& $<0.1$	&     35.5	&     1511	& 1358.000&140.6	& 0	& $<0.1$	&      7.6	&        3\\
&&&50               	& 1040.400&51.0	& 3	& $<0.1$	&     42.4	&     2242	& 1040.400&116.4	& 0	& $<0.1$	&      6.1	&        4\\
&&50&2              	& 2214.527&2323.9	& 33	& $<0.1$	&    590.2	&    74790	& 2214.527&239.0	& 33	& $<0.1$	&    125.3	&       12\\
&&&5                	& 1900.510&2006.1	& 31	& $<0.1$	&    308.2	&    28631	& 1900.510&140.1	& 21	& $<0.1$	&     71.6	&        9\\
&&&10               	& 1675.600&584.4	& 3	& $<0.1$	&    108.6	&    13344	& 1675.600&35.0	& 0	& $<0.1$	&      9.3	&        3\\
&&&50               	& 1358.000&536.4	& 3	& $<0.1$	&    105.4	&    12981	& 1358.000&41.9	& 3	& $<0.1$	&      8.9	&        3\\ \hline
%\pagebreak
800&500&2&2         	& 2184.000&18.6	& 3	& $<0.1$	&     16.6	&       70	& 2184.000&316.8	& 0	& $<0.1$	&      9.3	&        3\\
&&&5                	& 1442.152&808.3	& 64	& 0.7	&    790.2	&     4257	& 1442.152&5182.2	& 85	& 0.7	&    323.7	&       38\\
&&&10               	& 1126.576&100.9	& 12	& $<0.1$	&     97.9	&      840	& 1126.576&1991.1	& 9	& $<0.1$	&     27.3	&       12\\
&&&50               	& 821.152&283.8	& 21	& 0.2	&    280.1	&     1284	& 821.152&810.7	& 3	& $<0.1$	&     20.2	&       14\\
&&5&2               	& 2914.296&995.6	& 49	& $<0.1$	&    682.4	&    16441	& 2914.296&3469.2	& 35	& $<0.1$	&    107.4	&       15\\
&&&5                	& 2184.000&21.4	& 3	& $<0.1$	&     18.8	&      337	& 2184.000&205.7	& 0	& $<0.1$	&      9.9	&        3\\
&&&10               	& 1786.780&24.4	& 3	& $<0.1$	&     21.9	&      510	& 1786.780&188.1	& 0	& $<0.1$	&      9.4	&        4\\
&&&50               	& 1331.486&23.8	& 5	& $<0.1$	&     20.9	&      595	& 1331.486&202.3	& 3	& $<0.1$	&      9.3	&        7\\
&&10&2              	& 3240.524&225.4	& 23	& $<0.1$	&    196.7	&     5899	& 3240.524&1228.8	& 24	& $<0.1$	&    110.7	&       16\\
&&&5                	& 2581.220&42.0	& 3	& $<0.1$	&     35.3	&     1753	& 2581.220&125.8	& 0	& $<0.1$	&     11.0	&        4\\
&&&10               	& 2184.000&22.4	& 3	& $<0.1$	&     18.0	&      733	& 2184.000&78.9	& 0	& $<0.1$	&      7.9	&        3\\
&&&50               	& 1693.300&25.9	& 3	& $<0.1$	&     21.6	&     1010	& 1693.300&71.0	& 0	& $<0.1$	&      5.6	&        3\\
&&50&2              	& 3546.318& (0.0)	& 43	& $<0.1$	&    853.2	&    85399	& 3546.559&2213.6	& 572	& $<0.1$	&   1497.9	&       12\\
&&&5                	& 3036.491&1613.8	& 31	& $<0.1$	&    311.0	&    30928	& 3036.491&97.3	& 17	& $<0.1$	&     46.9	&        9\\
&&&10               	& 2674.700&698.3	& 9	& $<0.1$	&    171.5	&    18928	& 2674.700&51.1	& 7	& $<0.1$	&     18.4	&        3\\
&&&50               	& 2184.000&719.5	& 9	& $<0.1$	&    183.4	&    20945	& 2184.000&39.7	& 3	& $<0.1$	&     10.6	&        3\\
&800&2&2            	& 2184.000&32.5	& 3	& $<0.1$	&     30.5	&      111	& 2184.000&828.1	& 0	& $<0.1$	&     16.3	&        3\\
&&&5                	& 1415.621&3350.5	& 110	& 1.1	&   3231.9	&    12120	& 1411.821& (1.4)	& 11	& 1.4	&    124.9	&       22\\
&&&10               	& 1122.056&1682.8	& 46	& 0.5	&   1656.5	&     4751	& 1122.056& (0.5)	& 6	& 0.5	&     97.5	&       17\\
&&&50               	& 799.141&14.4	& 5	& $<0.1$	&     12.3	&       73	& 799.141&2394.6	& 3	& $<0.1$	&     16.1	&        8\\
&&5&2               	& 2933.936&5957.1	& 122	& $<0.1$	&   2777.6	&    37266	& 2931.925& (0.3)	& 41	& 0.2	&    260.7	&       16\\
&&&5                	& 2184.000&109.2	& 3	& $<0.1$	&     98.7	&     1659	& 2184.000&543.6	& 0	& $<0.1$	&     19.1	&        3\\
&&&10               	& 1778.004&7132.1	& 84	& 0.1	&   2625.8	&    44571	& 1778.004&2169.9	& 37	& $<0.1$	&    111.6	&       12\\
&&&50               	& 1306.711& (0.1)	& 33	& 0.1	&   2966.4	&    54900	& 1307.081&2668.7	& 43	& $<0.1$	&    116.9	&       37\\
&&10&2              	& 3237.819& (0.1)	& 43	& 0.1	&   1525.3	&    45013	& 3215.528& (0.8)	& 4	& 0.8	&     81.8	&       15\\
&&&5                	& 2588.917&837.6	& 27	& $<0.1$	&    494.3	&    10572	& 2588.917&2467.1	& 47	& $<0.1$	&    360.7	&       13\\
&&&10               	& 2184.000&44.7	& 3	& $<0.1$	&     37.2	&      887	& 2184.000&260.4	& 3	& $<0.1$	&     18.3	&        5\\
&&&50               	& 1664.814&119.0	& 5	& $<0.1$	&     81.9	&     2778	& 1664.814&262.7	& 5	& $<0.1$	&     16.1	&        7\\
&&50&2              	& 3568.798& (0.0)	& 18	& $<0.1$	&   1097.0	&   104982	& 3568.798&1080.0	& 86	& $<0.1$	&    353.1	&       12\\
&&&5                	& 3060.560& (0.0)	& 23	& $<0.1$	&    790.4	&    52487	& 3060.612&2656.5	& 174	& $<0.1$	&   1453.8	&       15\\
&&&10               	& 2703.168&1471.5	& 9	& $<0.1$	&    329.8	&    27390	& 2703.168&240.2	& 43	& $<0.1$	&    112.6	&        7\\
&&&50               	& 2184.000&1256.7	& 9	& $<0.1$	&    312.8	&    25238	& 2184.000&77.5	& 8	& $<0.1$	&     23.4	&        6\\
&1000&2&2           	& 2184.000&42.0	& 3	& $<0.1$	&     39.9	&      133	& 2184.000&1249.4	& 0	& $<0.1$	&     19.9	&        3\\
&&&5                	& 1415.597& (0.8)	& 78	& 0.9	&   6614.6	&    19353	& 1403.724& (3.4)	& 2	& 3.4	&     37.8	&        9\\
&&&10               	& 1121.391&3334.8	& 76	& 0.7	&   3229.3	&     8682	& 1116.095& (2.8)	& 2	& 2.8	&     28.6	&        7\\
&&&50               	& 797.457&96.3	& 5	& $<0.1$	&     94.0	&      230	& 797.457&4115.0	& 5	& $<0.1$	&     24.1	&       11\\
&&5&2               	& 2938.506&1728.0	& 65	& $<0.1$	&   1473.1	&     9301	& 2888.492& (1.8)	& 8	& 1.8	&    113.2	&       11\\
&&&5                	& 2184.000&60.3	& 3	& $<0.1$	&     53.9	&      702	& 2184.000&806.8	& 0	& $<0.1$	&     25.6	&        4\\
&&&10               	& 1782.909&6705.6	& 41	& $<0.1$	&   2306.6	&    31284	& 1782.909&3180.6	& 50	& $<0.1$	&    222.7	&       18\\
&&&50               	& 1308.218&120.3	& 5	& $<0.1$	&    101.0	&     1355	& 1308.218&1182.5	& 3	& $<0.1$	&     24.8	&       10\\
&&10&2              	& 3236.448& (0.2)	& 41	& 0.2	&   1505.9	&    36018	& 3216.068& (2.6)	& 2	& 2.6	&     62.1	&       11\\
&&&5                	& 2585.051&252.7	& 14	& 0.2	&    210.0	&     3068	& 2585.051&1022.5	& 11	& $<0.1$	&     96.3	&        6\\
&&&10               	& 2184.000&61.5	& 3	& $<0.1$	&     48.6	&     1077	& 2184.000&446.8	& 3	& $<0.1$	&     22.4	&        6\\
&&&50               	& 1661.573&55.6	& 3	& $<0.1$	&     41.9	&     1142	& 1661.573&352.7	& 0	& $<0.1$	&      8.0	&        3\\
&&50&2              	& 3570.485&5465.8	& 29	& $<0.1$	&    980.0	&    73490	& 3570.485&1156.2	& 55	& $<0.1$	&    298.4	&       12\\
&&&5                	& 3059.782&1224.6	& 3	& $<0.1$	&    291.0	&    22085	& 3059.782&153.1	& 1	& $<0.1$	&     34.6	&        6\\
&&&10               	& 2706.427&1099.6	& 3	& $<0.1$	&    260.7	&    20199	& 2706.427&111.9	& 6	& $<0.1$	&     29.1	&        6\\
&&&50               	& 2184.000&891.4	& 0	& $<0.1$	&    224.0	&    17672	& 2184.000&53.3	& 0	& $<0.1$	&      7.8	&        3\\ \hline
%\pagebreak
1000&500&2&2        	& 2716.000&18.4	& 1	& $<0.1$	&     16.3	&       38	& 2716.000&568.5	& 0	& $<0.1$	&     12.3	&        3\\
&&&5                	& 1760.418&1859.3	& 88	& 1.3	&   1775.0	&     9537	& 1760.118& (1.1)	& 49	& 3.1	&    308.0	&       34\\
&&&10               	& 1398.562&316.1	& 21	& 0.4	&    311.5	&     1512	& 1398.562&5688.7	& 13	& 0.2	&    113.2	&       24\\
&&&50               	& 1002.031&138.2	& 10	& $<0.1$	&    135.7	&      665	& 1002.031&1213.5	& 5	& $<0.1$	&     37.8	&       19\\
&&5&2               	& 3647.576&522.8	& 33	& $<0.1$	&    444.6	&     8293	& 3647.576&6271.7	& 48	& $<0.1$	&    250.5	&       15\\
&&&5                	& 2716.000&118.0	& 3	& $<0.1$	&    100.7	&     2497	& 2716.000&719.0	& 3	& $<0.1$	&     24.4	&        6\\
&&&10               	& 2220.333&194.8	& 7	& $<0.1$	&    157.7	&     4021	& 2220.333&810.6	& 3	& $<0.1$	&     28.6	&        7\\
&&&50               	& 1637.317&1193.7	& 30	& $<0.1$	&    763.6	&    17431	& 1637.317&2467.9	& 18	& 0.1	&     72.5	&       22\\
&&10&2              	& 4030.023&5298.1	& 44	& $<0.1$	&   1206.1	&    47553	& 4030.023&5788.0	& 45	& $<0.1$	&    215.3	&       12\\
&&&5                	& 3211.667&1410.5	& 20	& $<0.1$	&    555.5	&    17555	& 3211.667&729.6	& 15	& $<0.1$	&     82.4	&        7\\
&&&10               	& 2716.000&70.9	& 3	& $<0.1$	&     53.6	&     1846	& 2716.000&319.0	& 3	& $<0.1$	&     19.7	&        5\\
&&&50               	& 2079.311&1168.7	& 17	& $<0.1$	&    397.2	&    17227	& 2079.311&466.1	& 33	& $<0.1$	&     63.2	&       10\\
&&50&2              	& 4429.355& (0.0)	& 24	& $<0.1$	&   1322.3	&   138158	& 4429.497&1015.6	& 93	& $<0.1$	&    358.9	&       12\\
&&&5                	& 3792.352& (0.1)	& 7	& $<0.1$	&    827.5	&    71958	& 3793.855&3075.0	& 350	& $<0.1$	&   1692.5	&       11\\
&&&10               	& 3352.524& (0.0)	& 15	& $<0.1$	&    568.5	&    45772	& 3352.689&334.8	& 50	& $<0.1$	&    206.6	&       11\\
&&&50               	& 2715.420& (0.0)	& 18	& $<0.1$	&    647.9	&    47960	& 2716.000&250.6	& 90	& $<0.1$	&    166.5	&        7\\
&800&2&2            	& 2716.000&21.0	& 5	& 1.6	&     19.0	&       69	& 2716.000&1538.4	& 0	& $<0.1$	&     17.5	&        3\\
&&&5                	& 1765.264&5231.5	& 129	& 1.4	&   5001.4	&    14850	& 1750.855& (2.2)	& 3	& 2.2	&     77.7	&       15\\
&&&10               	& 1406.534&1918.3	& 56	& 0.5	&   1868.1	&     6872	& 1395.197& (2.1)	& 2	& 2.1	&     34.5	&        7\\
&&&50               	& 1005.860&1385.6	& 41	& 0.4	&   1363.7	&     4044	& 1003.198& (0.4)	& 5	& 0.4	&     38.7	&       14\\
&&5&2               	& 3642.011&2523.3	& 51	& $<0.1$	&   1522.2	&    17390	& 3563.581& (3.3)	& 1	& 3.3	&     26.2	&        6\\
&&&5                	& 2716.000&381.0	& 24	& $<0.1$	&    358.6	&     1968	& 2716.000&1544.2	& 27	& $<0.1$	&    173.9	&        8\\
&&&10               	& 2238.829&395.6	& 19	& $<0.1$	&    352.9	&     3027	& 2238.829&1164.9	& 11	& $<0.1$	&     70.4	&        8\\
&&&50               	& 1656.940&116.9	& 7	& $<0.1$	&    100.7	&     1818	& 1656.940&1421.9	& 5	& $<0.1$	&     37.3	&       10\\
&&10&2              	& 4018.780& (0.0)	& 47	& $<0.1$	&   1797.6	&    40779	& 3987.321& (1.7)	& 1	& 1.7	&     43.3	&        9\\
&&&5                	& 3193.171&258.5	& 8	& $<0.1$	&    181.3	&     3944	& 3193.171&2090.1	& 3	& $<0.1$	&     43.7	&        6\\
&&&10               	& 2716.000&1076.7	& 18	& $<0.1$	&    438.0	&     9351	& 2716.000&1161.2	& 13	& $<0.1$	&     87.7	&        5\\
&&&50               	& 2077.281&3746.6	& 32	& $<0.1$	&    969.9	&    23982	& 2077.281&1972.6	& 29	& $<0.1$	&     93.0	&       11\\
&&50&2              	& 4424.054& (0.0)	& 25	& $<0.1$	&   1575.7	&   119031	& 4424.300&3424.6	& 189	& $<0.1$	&   1104.0	&        8\\
&&&5                	& 3775.055&4062.1	& 15	& $<0.1$	&    586.2	&    35188	& 3775.055&231.7	& 11	& $<0.1$	&     70.6	&        4\\
&&&10               	& 3354.719&1451.2	& 3	& $<0.1$	&    342.1	&    23867	& 3354.719&135.2	& 3	& $<0.1$	&     50.0	&        3\\
&&&50               	& 2716.000&1218.2	& 3	& $<0.1$	&    345.4	&    25455	& 2716.000&81.5	& 0	& $<0.1$	&     11.1	&        3\\
&1000&2&2           	& 2716.000&22.2	& 3	& $<0.1$	&     20.1	&       58	& 2716.000&2392.7	& 0	& $<0.1$	&     24.3	&        3\\
&&&5                	& 1765.431& (1.1)	& 78	& 1.3	&   6879.6	&    17564	& 1763.068& (3.2)	& 7	& 3.2	&     74.5	&       12\\
&&&10               	& 1406.499&7135.8	& 85	& 4.6	&   6595.1	&    20373	& 1394.502& (74.3)	& 0	& 74.3	&     24.3	&        4\\
&&&50               	& 1004.072&3035.0	& 43	& 0.2	&   2971.3	&     8951	& 1001.851& (1.8)	& 1	& 1.8	&     23.0	&        8\\
&&5&2               	& 3642.702&2230.9	& 59	& $<0.1$	&   1787.2	&     9893	& 3575.072& (4.7)	& 1	& 4.7	&     31.5	&        5\\
&&&5                	& 2716.000&57.3	& 3	& $<0.1$	&     52.6	&      481	& 2716.000&1185.4	& 0	& $<0.1$	&     33.6	&        4\\
&&&10               	& 2240.512&85.5	& 3	& $<0.1$	&     78.4	&      953	& 2240.512&1040.4	& 3	& $<0.1$	&     34.0	&        5\\
&&&50               	& 1655.092&856.2	& 16	& $<0.1$	&    699.9	&     8243	& 1655.092&2678.5	& 9	& $<0.1$	&     74.4	&       17\\
&&10&2              	& 4019.055& (0.0)	& 36	& $<0.1$	&   1954.0	&    34314	& 3988.899& (1.7)	& 1	& 1.7	&     36.1	&        6\\
&&&5                	& 3191.488&161.4	& 3	& $<0.1$	&    113.5	&     2445	& 3191.488&2520.4	& 3	& $<0.1$	&     47.3	&        6\\
&&&10               	& 2714.968& (0.1)	& 43	& 0.1	&   1445.5	&    22784	& 2716.000&2051.7	& 37	& $<0.1$	&    248.0	&        6\\
&&&50               	& 2074.701&3887.3	& 17	& $<0.1$	&   1035.0	&    25190	& 2074.701&1802.8	& 19	& $<0.1$	&     60.6	&        8\\
&&50&2              	& 4425.640& (0.0)	& 10	& $<0.1$	&   1451.7	&    95933	& 4425.624& (0.0)	& 246	& $<0.1$	&   2009.2	&       12\\
&&&5                	& 3774.285& (0.1)	& 4	& $<0.1$	&    688.5	&    40884	& 3776.152&2104.5	& 87	& $<0.1$	&    702.1	&        8\\
&&&10               	& 3357.253& (0.0)	& 21	& $<0.1$	&    682.6	&    34138	& 3357.274&423.9	& 29	& $<0.1$	&    257.6	&        4\\
&&&50               	& 2716.000&1230.4	& 3	& $<0.1$	&    297.6	&    18523	& 2716.000&115.1	& 1	& $<0.1$	&     17.9	&        3\\
\end{longtable}

\end{document}